\documentclass[final]{siamltex}
\usepackage{amssymb,amsmath,amsfonts}
\usepackage{graphicx}
\usepackage{enumerate}
\usepackage{mathtools}

\usepackage{datetime}
\usepackage[usenames]{color}
\usepackage{hyperref}

\usepackage{xspace}
\DeclareMathOperator{\diam}{diam}
\newcommand{\ve}{{\textup{\textsf{v}}}}
\numberwithin{equation}{section}
\numberwithin{theorem}{section}
\newtheorem{remark}[theorem]{Remark}
\catcode`\@=11
\@addtoreset{equation}{section}
\catcode`\@=11
\@addtoreset{table}{section}
\numberwithin{table}{section}
\catcode`\@=11
\@addtoreset{figure}{section}
\numberwithin{figure}{section}
\DeclareMathOperator*{\signum}{sgn}

\DeclareMathOperator*{\esssup}{ess\;sup}

\newcommand{\SCAL}{{\cdot}}

\newcommand{\DIV}{\nabla\!{\cdot}}   
\newcommand{\GRAD}{\nabla}           

\def\Ldeux{{{  L}^2   (\Omega)}}

\def\Hsp1d{{{   \bf H}^{s+1}   (\Omega)}}

\def\Hdeux{{{   H}^2   (\Omega)}}

\def\BV{BV(\Omega)}
\newcommand{\Real}{\mathbb R}
\newcommand{\diff}{\, \mbox{\rm d}}
\newcommand{\vare}{{\varepsilon}}
\newcommand{\dt}{{\Delta t}}

\newcommand{\ie}{i.e.,\@\xspace}
\newcommand{\eg}{e.g.\@\xspace}
\newcommand{\cf}{cf.\@\xspace}

\renewcommand{\ae}{a.e.\@\xspace}

\newcommand{\ue}{\textup{\textsf{u}}}

\newcommand{\De}{{\textup{D}}}

\def\scl{\left\langle}
\def\scr{\right\rangle}

\newcommand{\bn}{{\bf n}}

\newcommand{\bq}{{\bf q}}
\newcommand{\br}{{\bf r}}

\newcommand{\bz}{{\bf z}}

\newcommand{\bE}{{\bf E}}

\newcommand{\bL}{{\bf L}}

\newcommand{\bX}{{\bf X}}

\newcommand{\blambda}{{\boldsymbol \lambda}}

\newcommand{\calA}{{\mathcal A}}

\newcommand{\calC}{{\mathcal C}}
\newcommand{\calD}{{\mathcal D}}

\newcommand{\calH}{{\mathcal H}}
\newcommand{\calI}{{\mathcal I}}

\newcommand{\calN}{{\mathcal N}}
\newcommand{\calO}{{\mathcal O}}
\newcommand{\calP}{{\mathcal P}}

\newcommand{\calT}{{\mathcal T}}
\newcommand{\calU}{{\mathcal U}}

\newcommand{\polN}{{\mathbb N}}

\newcommand{\polP}{{\mathbb P}}
\newcommand{\polQ}{{\mathbb Q}}

\newcommand{\polV}{{\mathbb V}}

\newcommand{\frakB}{{\mathfrak B}}

\newcommand{\frakF}{{\mathfrak F}}

\newcommand{\frakI}{{\mathfrak I}}

\newcommand{\frakd}{{\delta}}
\newcommand{\frakh}{{\mathfrak h}}

\newcommand{\hu}{\widehat{u}}

\newcommand{\hphi}{\widehat{\phi}}


\title{Discrete Total Variation Flows Without Regularization\thanks{RHN and AJS are partially supported by
NSF grants DMS-0807811 and DMS-1109325. AJS is also partially supported by NSF grant DMS-1008058 and
an AMS-Simons grant.}}

\author{
S\"oren Bartels\thanks{Abteilung f\"ur Angewandte Mathematik,
Albert-Ludwigs-Universit\"at Freiburg
Hermann-Herder-Str.~10
79104 Freiburg i.Br., Germany.
\texttt{bartels@mathematik.uni-freiburg.de}}
\and
Ricardo H.~Nochetto\thanks{Department of Mathematics and Institute for Physical Science and Technology,
University of Maryland, College Park, MD 20742, USA.
\texttt{rhn@math.umd.edu}}
\and
Abner J.~Salgado\thanks{Department of Mathematics, University of Maryland, College Park, MD 20742, USA.
\texttt{abnersg@math.umd.edu}}
}

\begin{document}

\maketitle

\begin{abstract}
We propose and analyze an algorithm for the solution of the $L^2$-subgradient flow of the 
total variation functional. The algorithm involves no regularization, thus the numerical solution 
preserves the main features that motivate practitioners to consider this type of energy.
We propose an iterative scheme for the solution of the arising problems, show that the
iterations converge, and develop a stopping criterion for them.
We present numerical experiments which illustrate the power of the method,
explore the solution behavior, and compare with regularized flows.
\end{abstract}



\date{\textcolor{magenta}{Draft Version of \today, \currenttime.}}

\section{Introduction}
\label{sec:intro}
This work is concerned with the approximation of solutions to the following (formal)
initial boundary value problem:
\begin{equation}
  \begin{dcases}
    \ue_t = \DIV \left(
      \frac{\GRAD \ue }{\left|\GRAD \ue \right|} \right), & \text{in } \Omega \times (0,T), \\
      \frac{\GRAD \ue }{\left|\GRAD \ue \right|} \SCAL \bn = 0, \   \text{on } \partial\Omega, &
    \ue|_{t=0} = \ue_0, \  \text{ in } \Omega.
  \end{dcases}
\label{eq:tvflowstrong}
\end{equation}
Here $\Omega$ is an open, bounded and connected subset of $\Real^d$ with $d\geq1$;
$\partial\Omega$ denotes the boundary of $\Omega$ and $\bn$ its exterior unit normal; and
$T>0$ is a positive and finite time.
This equation and related ones are commonly known as \emph{very singular diffusion equations}
(see \cite{MR2746654,MR1865089}) since in flat regions, \ie $|\GRAD \ue |=0$, the diffusion is so strong that 
it is not a local effect anymore.

Beginning with the seminal paper \cite{Rudin1992259} equations of this class have
received considerable attention from the image processing community, since such models
preserve discontinuities while removing noise and other artifacts (see also
\cite{MR2338487,MR2790462,MR1472196,MR2575049,MR2821259}).
In addition, such equations appear in the modeling of
grain boundary motion \cite{Kobayashi2000141}; facet formation and evolution \cite{PhysRevB.78.235401};
electromigration \cite{Fu1997259} and various other problems stemming from materials science.
This clearly shows that the development of efficient and accurate numerical schemes for the
solution of this class of problems is of extreme importance.
To the best of our knowledge however, the techniques advocated in the literature 
for the solution of these equations as a rule involve a regularization somewhat related
with replacing the singular term by
\begin{equation}
\label{eq:regularization}
  |\GRAD \ue|_\epsilon = \sqrt{ \epsilon^2 + |\GRAD \ue|^2 }, \quad \epsilon > 0;
\end{equation}
see, for instance, \cite{MR2197709,MR2520163,MR1994316,MR2194526,MR2733098}. The disadvantages of 
this approach are twofold:
first, although these methods have been shown to converge, there is no clear understanding of the
relation between the regularization parameter $\epsilon$ and the discretization parameters;
second, the regularization destroys certain fundamental features of the solutions which
motivate the introduction of such models in the first place.

A notable exception to the trend mentioned above is the work \cite{BartelsTV} which,
inspired by the ideas advanced in \cite{MR2782122}, develops a finite element scheme for
total variation minimization which involves no regularization. In this work we adapt and extend the
ideas presented in \cite{BartelsTV} to the study of total variation flows. We propose and
analyze an unconditionally stable and convergent discretization scheme for the approximation 
of \eqref{eq:tvflowstrong}. In addition, we study an iterative scheme for the solution of
the discrete problems and develop an \emph{a posteriori} error estimator that provides a robust 
stopping criterion for the iterative scheme that guarantees that, although we only have approximate
solutions, the convergence properties of the method are not affected.

This work is organized as follows. The notation and conventions are set in Section~\ref{sec:notation}.
In \S\ref{sub:approxBV} we recall the definition and main properties of functions of bounded
variation and, in addition, we present an approximation result for these functions. This will
be our workhorse during the derivation of error estimates. The proper framework to
understand \eqref{eq:tvflowstrong}, \ie total variation flow and its properties are described in
Section~\ref{sec:TVflow}. Time and space discretization are discussed in Section~\ref{sec:discretization},
where we begin by reviewing results on implicit semidiscretizations of gradient flows in Hilbert spaces.
Then we provide a general theory for fully discrete
subgradient flows in Hilbert spaces, which we later apply to our problem. One of the salient novelties of this
general theory is that we allow for modifications in the energy which can be used, for instance, to take into
account the effects of quadrature.
In Section~\ref{sec:soln} we propose an iterative scheme for the
solution of the fully discrete problems, and devise a stopping criterion for the iterations,
which guarantees that the convergence orders are not damaged.
Section~\ref{sec:Numerics} contains
a series of numerical experiments, which illustrate and extend the properties and theory for
the developed scheme.
Finally, in Section~\ref{sec:TVmin}, we apply the approximation result of \S\ref{sub:approxBV}
to the problem of total variation minimization and prove convergence estimates similar to those
available in the literature, but with reduced regularity assumptions.

\section{Notation and Preliminaries}
\label{sec:notation}
We denote by $\Omega$ a bounded domain in $\Real^d$ with $d\geq1$.
The boundary of $\Omega$ is denoted by $\partial\Omega$ and we assume that 
$\partial\Omega \in \calC^{0,1}$. We set $T>0$ to be a finite final time. As usual, 
we denote by $L^p(\Omega)$ the space of Lebesgue integrable functions with exponent $p \in [1,\infty]$
and by $W^s_p(\Omega)$, $s\in\Real$, the usual Sobolev spaces. Spaces of 
vector valued functions and their elements will be represented with boldface characters.
Recall that the following interpolation inequality holds (\cf \cite{MR1681462})
\begin{equation}
  \| w \|_{L^r} \leq \|w \|_{L^p}^s \| w \|_{L^q}^{1-s},
  \quad p,q \in [1,\infty],\  \frac1{r} = \frac{s}{p} + \frac{1-s}{q}, \ s \in [0,1].
\label{eq:Lpinterpolation}
\end{equation}

Whenever $E$ is a normed space, we denote its norm by $\|\cdot\|_E$
and its dual by $E'$. 
The $L^2$-inner product will be denoted by $\scl\cdot,\cdot\scr$. Function spaces of vector-valued
functions will be denoted by boldface characters.
For function spaces, if
it is clear from the context, we will omit the domain of definition.
We denote by $\frakB_1(E)$ the unit ball
in $E$, \ie the set
$
  \frakB_1(E) = \left\{ x \in E: \|x\|_E \leq 1 \right\}.
$
Let $\phi : [0,T] \rightarrow E$ be a measurable function in the Bochner sense,
then, for $p\in[1,\infty]$, we define 
\[
  \| \phi \|_{L^p(E)}^p = \int_0^T \| \phi(t) \|_E^p \diff t, \  p < \infty,
  \qquad \| \phi \|_{L^\infty(E)} = \esssup_{t\in[0,T]} \| \phi(t) \|_E.
\]
We will denote by $\tfrac{\diff \phi}{\diff t}$ the time derivative of $\phi$.

To deal with time discretization, we introduce a time-step $\dt>0$, for simplicity assumed constant.
Then we partition the time interval via $t^k = k \dt$ with $k=\overline{0,K}$ and
$K=\lceil T/\dt \rceil$.
We use the notation $ \phi^\dt = \left\{ \phi^k \right\}_{k=0}^K$,
introduce the time increment operator
\begin{equation}
  \frakd \phi^k = \phi^k - \phi^{k-1},
\label{eq:defoffrakd}
\end{equation}
and the extrapolation operator
\begin{equation}
  \phi^{\star,k+1} = \phi^0, \ k = 0, \qquad \phi^{\star,k+1} = \phi^k + \frakd \phi^k, \ k>0.
\label{eq:defofstar}
\end{equation}
For $\phi^\dt \subset E$ and $p \in [1,\infty]$, we introduce the (semi)norms
\[
  \| \phi^\dt \|_{\ell^p(E)}^p = \dt \sum_{k=0}^K \| \phi^k \|_E^p,
  \quad
  \| \phi^\dt \|_{\frakh^{1/2}(E)}^2 = \sum_{k=1}^K \| \frakd \phi^k \|_E^2,
\]
with the usual modification for $p=\infty$.
Given a sequence $\phi^\dt \subset E $ we will want to be able to compare it to
functions defined on $[0,T]$. To this end, we define the Rothe interpolant $\hphi$,
that is, the piecewise linear function
\begin{equation}
  \widehat\phi(t) = \frac{t-t^k}\dt \phi^{k+1} + \frac{t^{k+1}-t}\dt \phi^k, \quad t\in[t^k,t^{k+1}].
\label{eq:defofhatU}
\end{equation}
Notice that, by construction, $\hphi \in \calC^{0,1}(E)$, and 
$\|\hphi\|_{L^\infty(E)} = \| \phi^\dt \|_{\ell^\infty(E)}$.
Recall also the well-known summation by parts formula: for every $\phi^\dt, \psi^\dt \subset \Ldeux$,
\begin{equation}
  \sum_{k=0}^{K-1} \left(\scl \phi^k, \frakd \psi^{k+1} \scr +
  \scl \frakd \phi^{k+1}, \psi^{k+1} \scr \right)
  = \scl \phi^K, \psi^K \scr - \scl \phi^0, \psi^0 \scr.
\label{eq:abel}
\end{equation}

We will carry out the space discretization with finite element techniques. In other
words, given $\Omega$ we introduce a so-called triangulation $\calT_h = \{T\}$ as a collection of cells
that satisfy the usual conformity and shape regularity assumptions \cite{Ci78,MR2050138}, and
are such that $ \bar\Omega = \bigcup_{T \in \calT_h} \bar T$.
We parametrize our collection of triangulations via $h = \max\left\{ \diam(T): T \in \calT_h \right\}$.
For simplicity, we assume that each $T$ is the isoparametric image of a so-called reference cell,
which can be either $\widehat T = [-1,1]^d$,
in which case we call the cells cubic; or
$
  \widehat T = \left\{ (x_1,\ldots,x_d) \in \Real^d : \ x_i \geq 0 \ \sum_{i=1}^d x_i \leq 1 \right\},
$
which are called simplices. We denote by $\{z_{j,T}\}$ the vertices of the cell $T$ and set 
$\calN_T = \# \{z_{j,T}\}$. Clearly, $\calN_T = 2^d$ for cubes and $\calN_T = d+1$ for simplices.

We define the finite element space
\begin{equation}
  \polV_h = \left\{ w_h \in \calC^0(\bar\Omega): v_h|_T \in \calP \right\} \subset W^1_\infty(\Omega),
\label{eq:defofVh}
\end{equation}
where $\calP = \polQ_1$ for cubes and $\calP = \polP_1$ for simplices. Here
$\polQ_1$ denotes the space of polynomials of degree at most one in each variable and $\polP_1$ the
space of polynomials of total degree not greater than one.

For a function $w$ such that $w|_T \in \calC^0(\bar T)$,
we define its local Lagrange interpolant $\calI_h w$ by
\begin{equation}
  \calI_h w |_T \in \calP :\ \calI_h w|_T(z_{j,T}) = w|_T(z_{j,T}),
  \quad j=\overline{1,\calN_T},
  \quad \forall T \in \calT_h.
\label{eq:disclagrange}
\end{equation}
This operator satisfies
\begin{equation}
  \| w - \calI_h w \|_{L^p(T)} + h \| \GRAD( w - \calI_h w ) \|_{L^p(T)} \leq c h^2 \| D^2 w \|_{L^p(T)},
\label{eq:lagrange}
\end{equation}
see, \eg \cite{Ci78}, for a proof.
We remark also that if $w \in \calC^0(\bar\Omega)$, then $\calI_h w \in \polV_h$.
It will be also necessary to introduce the Cl\'ement interpolant (\cf \cite{MR0400739})
$\Pi_h : L^1(\Omega) \rightarrow \polV_h$. This operator enjoys approximation properties similar to
\eqref{eq:lagrange}, the only difference being that the domain on the right hand side is a neighborhood of
$T$. More importantly, the operator is stable under any Sobolev norm, \ie
\begin{equation}
  \| \Pi_h w \|_{W^s_p} \leq c \| w \|_{W^s_p}, \quad s\geq0, \ p \in [1,\infty].
\label{eq:clementbdd}
\end{equation}

We will denote by $c$ a constant whose value might change at each occurrence.

\subsection{Functions of Bounded Variation and their Approximation}
\label{sub:approxBV}
We say that a function $w\in L^1(\Omega)$ belongs to the space $\BV$ (is of bounded variation)
if its derivative, $\De w$, in the sense of distributions
is a Radon measure. In other words, $|\De w|(\Omega) < \infty$, where for any Borel set $A\subset\Omega,$
\[
  |\De w |(A) = \sup\left\{ \int_A w \DIV \bq: \ 
          \bq \in \calC_0^\infty(A), \, \|\bq \|_{\bL^\infty(A)} \leq 1 \right\}.
\]
The space $\BV$ endowed with the norm
$
  \| w \|_{BV} = \| w \|_{L^1} + |\De w |(\Omega)
$
is a Banach space. For more details on this space, we refer to \cite{MR1857292,MR1014685}.
Let us present a result on approximation by smooth functions, in the case of a star-shaped domain.

\begin{proposition}[Approximation of $BV$ functions]
\label{prop:approxbv}
Assume that $\Omega$ is bounded, star-shaped with respect to a point and $\partial\Omega\in \calC^{0,1}$. 
If $w\in\BV$, then for every $\epsilon>0$ there exists a $w_\epsilon \in \calC^\infty(\Omega)$
such that
\[
  \| w - w_\epsilon \|_{L^1} \leq \epsilon |\De w|(\Omega), \quad
  \| \GRAD w_\epsilon \|_{L^1} \leq (1 + c \epsilon)|\De w|(\Omega), \quad
  \| D^2 w_\epsilon \|_{L^1} \leq c\epsilon^{-1} |\De w|(\Omega).
\]
\end{proposition}
\begin{proof}
Without loss of generality, we can assume that $0\in\Omega$ and that $\Omega$ is star-shaped with respect to
$0$. For $\epsilon>0$ we define
\[
  \Omega_\epsilon = \left\{ y \in \Real^d: y=(1+\epsilon)x,\ x \in \Omega \right\},
\]
and notice that $\Omega$ and $\Omega_\epsilon$ are related via a bijective and Lipschitz, in fact linear,
transformation with Lipschitz constant $1+\epsilon$ and Jacobian $(1+\epsilon)^d \leq 1 + c \epsilon$.
For $w\in \BV$ we define $v_\epsilon\in BV(\Omega_\epsilon)$ via 
$v_\epsilon(y) = w(\tfrac{y}{1+\epsilon})$ and,
for $x\in\Omega$, $w_\epsilon(x) = v_\epsilon * \rho_\epsilon(x)$,
where $\rho_\epsilon$ is a smooth convolution kernel such that, for ever $1\leq p \leq \infty$,
$\|\GRAD \rho_\epsilon\|_{L^p} \leq c\epsilon^{-(1+d/p')}$ with $p'=p/(p-1)$.

If $w \in \calC^1(\Omega)$ then, clearly,
$\| w - w_\epsilon \|_{L^1} < \epsilon \| \GRAD w \|_{L^1} = \epsilon |\De w|(\Omega)$
and
\[
  \| \GRAD w_\epsilon \|_{L^1} \leq |\De v_\epsilon|(\Omega_\epsilon) 
  \leq (1+c\epsilon)\| \GRAD w \|_{L^1}
  = (1+c\epsilon)|\De w|(\Omega).
\]
We now recall that smooth functions are dense in $\BV$ under strict convergence, 
(\cf \cite{MR1857292,MR1014685}). In other words, given $w \in \BV$ there is a sequence
$\{ w_n \}_{n\in\polN} \subset \calC^\infty(\Omega)$ such that
\[
   \lim_{n\rightarrow \infty} \| w - w_n \|_{L^1} = 0,
   \qquad
   \limsup_{n\rightarrow \infty} |\De w_n |(\Omega) \leq | \De w |(\Omega).
\]
Applying the argument given above to elements of this sequence and then passing to the limit
we obtain the first two inequalities.

We use Young's inequality for convolutions \cite[Proposition~8.7]{MR1681462} to obtain
\[
  \| D^2 w_\epsilon \|_{L^1} \leq \|\GRAD w_\epsilon\|_{L^1} \|\GRAD \rho_\epsilon\|_{L^1}
  \leq c\frac{1+c\epsilon}{\epsilon}|\De w|(\Omega)
  \leq \frac{c}{\epsilon}|\De w|(\Omega),
\]
which concludes the proof.
\end{proof}

\begin{remark}[Approximation in $L^p$-spaces]
Observe that, in the setting of Proposition~\ref{prop:approxbv}, Young's inequality
also implies
\[
  \| D^2 w_\epsilon \|_{L^p} \leq \| \GRAD w_\epsilon \|_{L^1} \| \GRAD \rho_\epsilon \|_{L^p}
  \leq c \epsilon^{-(1+d/p')} |\De w|(\Omega),
\]
for any $p\in[1,\infty]$ with $p'=p/(p-1)$. In the sequel, however, we shall avoid using this bound
since it would lead to $d$-dependent error estimates.
\end{remark}

\section{The Total Variation Flow}
\label{sec:TVflow}

Let us define the functional $\Psi : \Ldeux \rightarrow \Real$ by
\begin{equation}
\label{eq:defofEnergy}
  \Psi(w) = 
  \begin{dcases}
    |\De w|(\Omega), & w \in \Ldeux \cap \BV , \\
    + \infty, & w \in \Ldeux\setminus \BV,
  \end{dcases}
\end{equation}
It is not difficult to show that $\Psi$ is convex and lower semicontinuous. Then, one can
define the subdifferential of $\Psi$
(see \cite{MR2033382,MR2582280,MR1929886,MR0348562,MR1727362,MR2746654,MR1865089})
and study its subgradient flow, \ie we seek for a function
$\ue:[0,T] \rightarrow \Ldeux$ such that
\begin{equation}
\label{eq:subTVflow}
  \ue_t \in - \partial\Psi(\ue),
\end{equation}
or, equivalently,
\begin{equation}
\label{eq:TVflowvarineq}
  \scl \ue_t, \ue - w \scr + \Psi(\ue) - \Psi(w) \leq 0, \quad \forall w \in \Ldeux.
\end{equation}
It is in this sense that \eqref{eq:tvflowstrong} is going to be understood and analyzed.
Existence of solutions to \eqref{eq:subTVflow} can be obtained with the help of the theory of
maximal monotone operators, see \cite{MR2582280,MR0348562}.

\begin{remark}[Dirichlet boundary conditions]
The definition that we have provided corresponds to imposing Neumann boundary conditions
as in \eqref{eq:tvflowstrong}. The issue of how to impose Dirichlet boundary conditions
is a delicate one since the trace of a $\BV$ function is in 
$L^1(\partial\Omega)$ (\cf \cite{MR1857292,MR1014685}). In addition, the functional \eqref{eq:defofEnergy}
has linear growth, so that the only possible way to impose Dirichlet boundary conditions is with the 
relaxed energy
$\Psi(w) + \int_{\partial\Omega} |w - g |$;
see \cite{MR2033382,MR1301176} for details. The introduction of this additional non-differentiable
term greatly complicates the analysis. However, in the particular situation when $\Omega$ is 
convex, the boundary data is time independent and continuous, \ie 
$g(x,t)=g(x) \in \calC^0(\partial\Omega)$, and the initial data is compatible with the boundary data
in the sense that $\ue_0|_{\partial\Omega} = g$, then the solution $\ue$ to the subgradient
flow with the relaxed energy satisfies $\ue(\cdot,t)|_{\partial\Omega} = g(\cdot)$,
for all $t \in (0,T]$; see \cite[Lemma 4.1]{MR1301176}.
This can be realized, for instance, in the case of homogeneous Dirichlet boundary conditions
($g\equiv0$) and compactly supported initial data. Under this particular setting all the 
results we present will also hold.
\end{remark}

Reference \cite[Theorem 2.16]{MR2033382} shows that problem \eqref{eq:subTVflow} possesses a
$L^d$--$L^\infty$ regularizing effect, that is if
$\ue_0 \in L^d(\Omega)$, then $\ue(t) \in L^\infty(\Omega)$ for all $t>0$. Let us show that
solutions to this problem also satisfy a maximum principle.

\begin{theorem}[Maximum principle for TV flow]
\label{thm:maxprinciple}
Assume that $\ue_0 \in L^\infty(\Omega)$, then $\ue$, solution of
\eqref{eq:TVflowvarineq}, is such that $\ue \in L^\infty([0,T],L^\infty(\Omega))$ and
\[
  \| \ue \|_{L^\infty(L^\infty)} \leq \| \ue_0 \|_{L^\infty}.
\]
\end{theorem}
\begin{proof}
Since $\ue_0 \in L^\infty(\Omega)$ we can define
$ \calU_0 = \esssup_{x\in\Omega}\ue_0(x)$,
and $ \bar w(t) = \ue(t)\vee \calU_0$.
Since $\supp \De \bar w = \{x \in \Omega: \ \ue \geq \calU_0 \}$ and, on this set,
$\De \bar w = \De \ue$, we can conclude that $\bar w\in \Ldeux\cap\BV$
and $\Psi(\bar w) \leq \Psi(\ue)$.
Setting $w=\bar w$ in \eqref{eq:TVflowvarineq} we obtain
$
  \scl \ue_t, 0 \vee (\ue- \calU_0) \scr \leq 0,
$
which, since $\calU_0$ is constant, implies
$\tfrac{\diff}{\diff t} \| (\ue - \calU_0) \vee 0  \|_{L^2}^2 \leq 0$.
Given that $(\ue_0 - \calU_0)\vee 0 = 0$ this implies the result.
\end{proof}

\section{Discretization}
\label{sec:discretization}

In this section we introduce and analyze a fully discrete scheme for the approximation of
solutions to \eqref{eq:TVflowvarineq}. 
We begin with a semidiscrete (continuous in space and discrete in time) scheme for 
\eqref{eq:TVflowvarineq} which can then be analyzed using standard results from the
literature (\cf \cite{MR1737503,MR1377244}).
Then we develop a theory for fully discrete subgradient flows in Hilbert spaces
and discuss the effect of introducing a discrete energy and a perturbation on the right hand
side. We will provide sufficient compatibility conditions between the space discretization and 
the discrete energy to guarantee convergence. These results constitute a general and, as far as we 
know, novel approach to the study of fully discrete schemes for subgradient flows
and evolution variational inequalities. The main application
of these results will be, of course, a fully discrete scheme for \eqref{eq:TVflowvarineq}.

\subsection{A Semidiscrete Scheme for TV Flows}
\label{sub:SemiDiscrete}
We introduce a sequence $\{u^\dt\}$ contained in $\Ldeux \cap\BV$
with $u^0 = \ue_0$ that solves:
\begin{equation}
  \scl \frac{\frakd u^{k+1} }\dt, u^{k+1}- w \scr + \Psi(u^{k+1}) - \Psi(w) \leq 0
  \quad \forall w \in \Ldeux.
\label{eq:semidiscvarineq}
\end{equation}
Existence and uniqueness is guaranteed by the convexity 
and lower semicontinuity of $\Psi$;
see \cite{MR2033382,MR2582280,MR0348562,MR1727362}.
\emph{A priori}
estimates, as well as a maximum principle are established in the next result.

\begin{proposition}[Semidiscrete stability]
\label{prop:semidiscretestable}
Let $\{u^\dt\}$ solve \eqref{eq:semidiscvarineq}. If $u^0 \in \Ldeux$, then
\begin{equation}
  \| u^\dt \|_{\ell^\infty(L^2)}^2 + \| u^\dt \|_{\frakh^{1/2}(L^2)}^2
  +  \dt \sum_{k=1}^K  |\De u^k |(\Omega) \leq c \| u^0 \|_{L^2}^2.
\label{eq:semidiscstab1}
\end{equation}
If $u^0 \in \Ldeux\cap\BV$, then the flow is monotone, \ie
$\Psi(u^{k+1}) \leq \Psi(u^k) \leq \Psi(u^0)$ for all $k\geq0$; and, moreover,
\begin{equation}
    \| u^\dt \|_{\frakh^{1/2}(L^2)}^2  + \dt \Psi(u^K) \leq \dt \Psi(u^0).
\label{eq:semidiscstab2}
\end{equation}
If $u^0 \in L^\infty(\Omega) \cap \BV$, then $u^\dt \subset L^\infty(\Omega)$ and
\[
  \| u^\dt \|_{\ell^\infty(L^\infty)} \leq \|u^0 \|_{L^\infty}.
\]
\end{proposition}
\begin{proof}
To obtain \eqref{eq:semidiscstab1} it suffices to set $w=0$ on \eqref{eq:semidiscvarineq}
and add over $k$. To obtain \eqref{eq:semidiscstab2} we set $w=u^k$ and add over $k$.
The maximum principle is obtained \emph{mutatis mutandis} the proof of Theorem~\ref{thm:maxprinciple}.
\end{proof}

The convergence properties of \eqref{eq:semidiscvarineq} are a consequence
of standard results \cite{MR1737503,MR1377244}. For convenience we summarize them below.

\begin{corollary}[Convergence of semidiscrete TV flow]
\label{cor:convsemidiscrete}
Assume that $u^0 \in \Ldeux\cap\BV$,
then
\[
  \| \ue - \hu \|_{L^\infty(L^2)}^2 + \dt \Psi(u^K) \leq 
  c \left( \| \ue_0 - u_0 \|_{L^2}^2 + \dt^{\alpha} \right),
\]
where $\hu$ is the piecewise linear function defined in \eqref{eq:defofhatU}
and $\alpha=1$.
If, in addition, $\partial\Psi(\ue_0) \cap \Ldeux \neq \emptyset$, then 
$\alpha = 2$.
\end{corollary}

\begin{remark}[Error estimates for semidiscrete flows]
\label{rem:besterror}
Notice that the conclusion of Corollary~\ref{cor:convsemidiscrete} gives an error estimate
of order $\calO(\dt^{1/2})$ under the sole assumption that the initial condition has finite
energy. Under this regularity the result is optimal, see \cite[Theorem 5]{MR1377244}.
Observe also that the assumption $\partial\Psi(\ue_0) \cap \Ldeux \neq \emptyset$, necessary
to obtain a $\calO(\dt)$ estimate, is not a regularity but rather a compatibility assumption.
We refer the reader to \cite{MR2033382} for a partial characterization of the subdifferential of
the total variation.
\end{remark}

\subsection{Fully Discrete Schemes for Subgradient Flows in Hilbert Spaces}
\label{sub:EulerandFEM}
Here we present and analyze of a fully discrete implicit Euler method for subgradient flows in Hilbert
spaces. The main novelty of our approach is the treatment of the space discretization and that
we allow for perturbations of the energy, as well as of the right hand side.

Let $H$  be a Hilbert space with inner product $\langle\cdot,\cdot\rangle$, which induces the
norm $\|\cdot\|_H$. Assume that the functional $\frakF : \calD(\frakF) \subset H \rightarrow \Real$
is convex, lower semicontinuous and $\overline{\calD(\frakF)} =H$.
Then, for every $w \in \calD(\frakF)$ the subdifferential 
$\partial \frakF(w) \subset H$ is not empty. For details we refer to 
\cite{MR2582280,MR0348562,MR1727362}. We want to study
its subgradient flow: Find $\ve : [0,T] \rightarrow H$ with 
$\ve|_{t=0} = \ve_0$, such that
\begin{equation}
  \ve_t + \partial \frakF(\ve) \ni 0,
\label{eq:subgradflow}
\end{equation}
or, equivalently,
\begin{equation}
  \scl \ve_t, \ve - w \scr + \frakF(\ve) - \frakF(w) \leq 0, \quad \forall w \in H.
\label{eq:subgradflowvarineq}
\end{equation}
The existence and uniqueness of a solution to \eqref{eq:subgradflow} or \eqref{eq:subgradflowvarineq}
follows the theory of maximal monotone operators, \cite{MR2582280,MR0348562}.

To discretize in time we consider the Euler method. In other words, we search for 
$v^\dt \subset H$, $v^0 = \ve_0$, such that
\begin{equation}
  \scl \frac{ \frakd v^{k+1} }\dt, v^{k+1} - w \scr + \frakF(v^{k+1}) - \frakF(w)
  \leq 0, \quad \forall w \in H.
\label{eq:discflowvarineq}
\end{equation}
The theory of \cite{MR1737503,MR1377244} provides, under the assumption that $v^0 \in \calD(\frakF)$
a $\calO(\dt^{1/2})$ error estimate as in Corollary~\ref{cor:convsemidiscrete}.

We now introduce the space discretization.
Let $\{\calH_h\}_{h>0}$ be a family of (finite dimensional) subspaces of $H$.
To be able to handle perturbations on the energy induced by the spatial
discretization we assume that, for each $h>0$, we have a convex and lower semicontinuous functional 
$\frakF_h: \calH_h \rightarrow \Real$ that is \emph{monotone}, in the sense that
\begin{equation}
  \frakF_h(w_h) \geq \frakF(w_h), \quad \forall w_h \in \calH_h.
\label{eq:discenergymonotone}
\end{equation}
Moreover, we assume that the spaces $\calH_h$ possess suitable approximation properties. In other words, 
there is a dense subspace $W \hookrightarrow H$, an operator $\calC_h : W \rightarrow \calH_h$ and functions
$\vare_i \in \calC \left( [0,\infty) , [0,\infty)\right)$, $\vare_i(0)=0$, $i=\overline{1,2}$, such that
\begin{equation}
  \| \calC_h w - w \|_H \leq \vare_1(h) \| w \|_W, \quad \forall w \in W,
\label{eq:interpolationabs}
\end{equation}
and this approximation is \emph{asymptotically energy diminishing}, \ie
\begin{equation}
  \frakF_h(\calC_h w) - \frakF(w) \leq \vare_2(h) \| w \|_W, \quad \forall w \in W.
\label{eq:interpolationabstracr}
\end{equation}

We consider the following fully discrete problem: Find $v_h^\dt \subset \calH_h$ such that
\begin{equation}
  \scl \frac{ \frakd v_h^{k+1} }\dt, v_h^{k+1} - w_h \scr + \frakF_h(v_h^{k+1}) - \frakF_h(w_h) 
  \leq \scl \rho^{k+1}, v_h^{k+1} - w_h \scr,
\label{eq:fulldiscEuler}
\end{equation}
for all $w_h \in \calH_h$, where $\rho^\dt \subset H$ is a perturbation.

We begin our analysis of the fully discrete method \eqref{eq:fulldiscEuler} with an
\emph{a priori} bound on the increments of the sequence $v_h^\dt$.

\begin{lemma}[Stability of derivatives]
\label{lem:stabderivh}
The solution $v_h^\dt$ to \eqref{eq:fulldiscEuler} satisfies
\[
  \frac12 \| v_h^\dt \|_{\frakh^{1/2}(H)}^2+ \dt \frakF_h(v_h^K)
  \leq   \dt \frakF_h(v_h^0) + \frac\dt2 \| \rho^\dt \|_{\ell^2(H)}^2.
\]
\end{lemma}
\begin{proof}
Set $w_h = v_h^k$ in \eqref{eq:fulldiscEuler}.
Multiplying by $\dt$, adding over $k$, and using Young's inequality, we obtain the result.
\end{proof}

The sequence $\rho^\dt$ is meant to be a perturbation induced by either discretization
or the solution procedure. For this reason we shall assume that
\begin{equation}
  \| \rho^\dt \|_{\ell^\infty(H)} \leq c \dt^{1/2}.
\label{eq:rhosmall}
\end{equation}
Based on this estimate, the error analysis proceeds as follows.

\begin{theorem}[A priori error analysis]
\label{thm:aprioribetter}
Let $v^\dt$ be the solution to \eqref{eq:discflowvarineq} and
$v_h^\dt$ the solution of \eqref{eq:fulldiscEuler}. Assume $v^\dt \in \ell^\infty(W)$.
If the operator $\calC_h$ satisfies 
\eqref{eq:interpolationabs} and \eqref{eq:interpolationabstracr};
the discrete energies $\frakF_h$ satisfy \eqref{eq:discenergymonotone};
$ \frakF_h(v_h^0) < + \infty$ uniformly in $h$; and the perturbations $\rho^\dt$ satisfy
\eqref{eq:rhosmall}, then
there exists a constant $c>0$ proportional to $T$ such that
\[
  \| v^\dt - v_h^\dt \|_{\ell^\infty(H)}^2 \leq \| v^0 - v_h^0 \|_H^2 
  + c (\vare_1(h) + \vare_2(h) )
  \| v^\dt \|_{\ell^\infty(W)} + c \dt.
\]
\end{theorem}
\begin{proof}
To simplify notation, as usual, we will denote $e^k = v^k - v_h^k$.
Set $w = v_h^{k+1}$ in \eqref{eq:discflowvarineq} and $w_h = \calC_h v^{k+1}$ in 
\eqref{eq:fulldiscEuler} and add the results. Using the monotonicity property 
\eqref{eq:discenergymonotone} we obtain
\begin{multline*}
  \scl \frac{ \frakd e^{k+1} }\dt, e^{k+1} \scr
  \leq
  \frakF_h(\calC_h v^{k+1}) - \frakF(v^{k+1})
  {- \scl \rho^{k+1}, e^{k+1} \scr} \\
  + \scl \rho^{k+1}, v^{k+1} - \calC_h v^{k+1} \scr
  +\scl \frac{ \frakd v_h^{k+1} }\dt, \calC_h v^{k+1} - v^{k+1} \scr .
\end{multline*}
Using Lemma~\ref{lem:stabderivh} and \eqref{eq:rhosmall} we obtain 
$\dt^{-1} \sum_{k=1}^K \| \frakd v_h^k \|_H^2 \leq c$, whence \eqref{eq:interpolationabs} and
\eqref{eq:interpolationabstracr} yield
\[
  \scl \frakd e^{k+1}, e^{k+1} \scr \leq c \dt
      \left[ \vare_2(h) + \vare_1(h) \right] \| v^{k+1} \|_W
      + \dt \| \rho^{k+1} \|_H \| e^{k+1} \|_H.
\] 
Since the sequence $\|e^\dt\|_H$ is finite we can assume that its maximum is
attained for $k=\kappa$. Summing the above inequality over $k=\overline{1,\kappa-1}$ we
deduce
\[
  \| e^\kappa \|_H^2 + \| \frakd e^\kappa \|_H^2 \leq
  \| e^0 \|_H^2 +
  c\left( \vare_1(h) + \vare_2(h) \right) \| v^\dt \|_{\ell^\infty(W)}
  + 2\dt \| e^\kappa \|_H \sum_{k=1}^\kappa \| \rho^k\|_H.
\]
An application of Young's inequality,
together with assumption \eqref{eq:rhosmall} allows us to conclude.
\end{proof}

\subsection{Fully Discrete Scheme for TV Flows}
\label{sub:fulldiscTV}

Let us specialize the ideas presented in \S\ref{sub:EulerandFEM} to the case of TV flows.
To do so, we will work on the discrete spaces \eqref{eq:defofVh}, define discrete energies
and approximation operators, and verify that assumptions
\eqref{eq:discenergymonotone}--\eqref{eq:interpolationabstracr}
are satisfied. The conclusion of Theorem~\ref{thm:aprioribetter} will then allow us to obtain
error estimates.

In this setting, we consider the subgradient flow: Find $u_h^\dt \subset \polV_h$ that solves
\begin{equation}
  \scl \frac{\frakd u_h^{k+1} }\dt, u_h^{k+1} - w_h \scr
  + \Psi_h(u_h^{k+1}) - \Psi_h(w_h) \leq 0, \quad \forall w_h \in \polV_h.
\label{eq:ourprobfulldiscvarineq}
\end{equation}
The discrete energy is defined by
\begin{equation}
  \Psi_h(w_h) = \sum_{T \in \calT_h } \int_T \calI_h |\GRAD w_h |,
\label{eq:defofdiscEnergy}
\end{equation}
where the operator $\calI_h$ was defined in \eqref{eq:disclagrange}.
Notice that we have effectively replaced the total variation seminorm by a quadrature formula. Indeed,
we can rewrite $\Psi_h$ as
\[
  \Psi_h(w_h) = \sum_{T \in \calT_h} \sum_{j=1}^{\calN_T} |\GRAD w_h|(z_{j,T}) 
  \int_T \lambda_{j,T}(z) \diff z
  = \sum_{T \in \calT_h} |T|\sum_{j=1}^{\calN_T} \omega_j  |\GRAD w_h|(z_{j,T}),
\]
where $\lambda_{j,T}$ is the coordinate basis function associated with node $z_{j,T}$
and we denote the weights by
$\omega_j = |T|^{-1}\int_T \lambda_{j,T}(z) \diff z$.
Notice that this modification fits into the framework described in \S\ref{sub:EulerandFEM}.
Its utility will become clear in the following paragraph.

The approximation operator is defined as
\begin{equation}
  \calC_h w = \Pi_h w_\epsilon,
\label{eq:defofCh}
\end{equation}
where $\Pi_h$ is the Cl\'ement interpolation operator
and $w_\epsilon$ denotes a regularization of $w$ as in 
Proposition~\ref{prop:approxbv}
with $\epsilon = h^{2/3}$. Proving error estimates reduces to verification of the hypotheses
of \S\ref{sub:EulerandFEM}.

\begin{corollary}[Convergence of TV flow]
\label{cor:aprioribetter}
Let $u^\dt$ be the solution of \eqref{eq:semidiscvarineq}
and $u_h^\dt$ the solution of \eqref{eq:ourprobfulldiscvarineq} with the discrete energy given in
\eqref{eq:defofdiscEnergy}.
If $\Omega$ is star shaped with respect to a point; $u^0 \in L^\infty(\Omega)\cap\BV$;
and $\Psi_h(u_h^0)\leq c < +\infty$ uniformly in $h$, then
\[
  \| u^\dt - u_h^\dt \|_{\ell^\infty(L^2)}^2 \leq \| u^0 - u_h^0 \|_{L^2}^2 + ch^{1/3}.
\]
\end{corollary}
\begin{proof}
Notice, first of all, that the given assumptions on
$u^0$ translate into
$u^\dt \in \ell^\infty(L^\infty(\Omega)\cap\BV)$. Consequently, we set
$W = \BV \cap L^\infty(\Omega)$ and
\[
  \| w \|_W = \| w \|_{L^\infty} + |\De w |(\Omega).
\]
It suffices to verify the abstract assumptions 
\eqref{eq:discenergymonotone}--\eqref{eq:interpolationabstracr}:
\begin{enumerate}
  \item[\eqref{eq:discenergymonotone}:] If $w_h|_T \in \polP_1$, then $\GRAD w_h$ is constant
  and $\calI_h |\GRAD w_h | = |\GRAD w_h |$. If $w_h|_T \in \polQ_1$ instead, then
  its gradient is linear
  and, consequently, $|\GRAD w_h|$ is convex. Then, we only need to realize that if the function
  $\varphi$ is convex, then $\calI_h \varphi|_T \geq \varphi|_T$. Indeed, using that
  $\sum_{j=1}^{\calN_T} \lambda_{j,T} \equiv 1$,
  \[
    \calI_h \varphi|_T(z) = \sum_{j=1}^{\calN_T} \varphi(z_{j,T}) \lambda_{j,T}(z)
    \geq \varphi\left( \sum_{j=1}^{\calN_T} z_{j,T} \lambda_{j,T}(z) \right) = \varphi(z).
  \]
  
  \item[\eqref{eq:interpolationabs}:] \eqref{eq:Lpinterpolation}, together with the
  $L^\infty$-stability of $\Pi_h$, gives
  \[
    \| \Pi_h w_\epsilon - w \|_{L^2}^2 \leq \| \Pi_h w_\epsilon - w \|_{L^1} 
    \| \Pi_h w_\epsilon - w \|_{L^\infty} \leq c \|w\|_{L^\infty} \| \Pi_h w_\epsilon - w \|_{L^1} .
  \]
  To bound the $L^1$ norm in the above inequality 
  we add and subtract $w_\epsilon$ to obtain
  \[
    \| \Pi_h w_\epsilon - w \|_{L^1} \leq 
    \| \Pi_h w_\epsilon - w_\epsilon \|_{L^1} + \| w_\epsilon - w \|_{L^1}
    \leq c \left( \frac{h^2}\epsilon + \epsilon \right) |\De w |(\Omega),
  \]
  where we have used \eqref{eq:lagrange}, for $p=1$, in conjunction with 
  Proposition~\ref{prop:approxbv}.

  \item[\eqref{eq:interpolationabstracr}:] We begin the proof of the energy diminishing property by
  \begin{multline*}
    \Psi_h(\Pi_h w_\epsilon) - \Psi(w) =
    \int_\Omega \calI_h |\GRAD \Pi_h w_\epsilon| - |\De w|(\Omega) \leq \\
    \left( \int_\Omega (\calI_h |\GRAD \Pi_h w_\epsilon| - |\GRAD \Pi_h w_\epsilon|) \right) +
    \left( \int_\Omega (|\GRAD \Pi_h w_\epsilon | - |\GRAD w_\epsilon|)  \right) + \\
    \left( \int_\Omega |\GRAD w_\epsilon | - |\De w|(\Omega)  \right)
    = I + II + III.
  \end{multline*}
  Applying Proposition~\ref{prop:approxbv} yields $III \leq c \epsilon |\De w|(\Omega)$. We next
  invoke the triangle inequality along with interpolation estimate \eqref{eq:lagrange} for
  $\Pi_h$ to arrive at
  \[
    II \leq \int_\Omega \left | \GRAD( \Pi_h w_\epsilon - w_\epsilon) \right|
    \leq ch \| D^2 w_\epsilon \|_{L^1}
    \leq c \frac{h}\epsilon |\De w |(\Omega).
  \]
  For the first term $I$, we use \eqref{eq:lagrange} to obtain
  \[
    I = \int_\Omega \calI_h |\GRAD \Pi_h w_\epsilon| - |\GRAD \Pi_h w_\epsilon|
    \leq c h \sum_{T \in \calT_h } \left\| \GRAD (|\GRAD\Pi_h w_\epsilon|) \right\|_{L^1(T)}.
  \]
  For a smooth function,
  \[
    \GRAD (|\GRAD f|) = \frac{ \GRAD f }{|\GRAD f |} D^2f,
    \quad \text{\ae} \  \left\{ x \in T: \ \GRAD f(x) \neq 0 \right\},
  \]
  and $\GRAD (|\GRAD f|) = 0$ \ae $\{ x \in T:\ \GRAD f = 0\}$. This allows us to conclude that
  $ |\GRAD (|\GRAD f|) | \leq |D^2f|$ \ae in $T$.
  This, together with the bound \eqref{eq:clementbdd} and Proposition~\ref{prop:approxbv},
  shows that
  \[
    \int_\Omega \calI_h |\GRAD \Pi_h w_\epsilon| - |\GRAD \Pi_h w_\epsilon|
    \leq c h \| D^2 w_\epsilon \|_{L^1} \leq c \frac{h}{\epsilon} |\De w|(\Omega).
  \]
\end{enumerate}

The estimates above allow us to see that
\[
  \vare_1(h) = c \left( \frac{h^2}\epsilon + \epsilon \right)^{1/2}
  \quad
  \vare_2(h) = c \left( \frac{h}\epsilon + \epsilon \right).
\]
Setting $\epsilon = h^{2/3}$ we obtain the result.
\end{proof}

\begin{remark}[Energy diminishing interpolation]
\label{rem:tvd}
The proof of Corollary~\ref{cor:aprioribetter} actually shows that if we were able to construct
a \emph{TV diminishing} interpolant, \ie such that
\[
  \int_\Omega |\GRAD \calC_h w | \leq |\De w|(\Omega),
\]
then $\vare_2(h) = ch\epsilon^{-1}$ and, setting $\epsilon = h^{1/2}$ we would improve the
rate of convergence to $\calO(h^{1/4})$. Under the assumption that
$\Omega =(0,1)^d$ and that the mesh is Cartesian, \cite{TVDInterpolation} presents such
a construction.
\end{remark}

\section{An Iterative Scheme for the Fully Discrete Scheme}
\label{sec:soln}
A practical algorithm for the solution of the discrete variational inequalities
\eqref{eq:ourprobfulldiscvarineq} would require identifying the elements of the 
subdifferential $\partial\Psi_h(u)$. In the continuous setting, relying on the results of
\cite{MR750538}, such identification is presented in \cite{MR2033382}.
Let us describe this identification without going into technical details. We introduce the space
\[
  \bX = \left\{ \bq \in \bL^\infty(\Omega): \ \DIV \bq \in \Ldeux,\ \bq\SCAL\bn =0 \right\},
\]
and stress that $w \in \partial\Psi(u)$ if and only if \cite{MR2033382}
\[
  \Psi(u) = \int_\Omega w u, \qquad \text{and} \qquad 
  \exists \bz \in \bX\cap \frakB_1(\bL^\infty(\Omega)):  \ w = -\DIV \bz.
\]
Notice that we can replace the first equality above by $\Psi(u) = \int_\Omega \bz\SCAL \De u$ 
\cite{MR2033382}. This serves as motivation for the discrete energy $\Psi_h$
given in \eqref{eq:defofdiscEnergy}. Indeed,
\begin{align*}
  \Psi_h(w_h) &=
   \sum_{T \in \calT_h } |T| \sum_{j=1}^{\calN_T} \omega_j |\GRAD w_h|(z_{j,T})
  = \sum_{T \in \calT_h } |T|\sum_{j=1}^{\calN_T} \omega_j \sup_{\bq \in \Real^d: |\bq|\leq 1}
  \GRAD w_h(z_{j,T})\SCAL \bq \\
  &= \sup_{\bq_h\in \bX_h:\ \|\bq_h\|_{\bX_h} \leq 1} \scl \GRAD w_h, \bq_h \scr_h
  = \sup_{\bq_h\in \bX_h} \left \{\scl \GRAD w_h, \bq_h \scr_h - \frakI_{\frakB_1(\bX_h)}(\bq_h) \right\},
\end{align*}
where $\frakI_S$ denotes the indicator function of $S$,
the discrete space $\bX_h$ is defined as
\[
  \bX_h = \left\{ \bq_h \in \bL^\infty(\Omega):\ \bq_h|_T \in \polQ_1^d, 
  \forall T \in \calT_h \right\},
  \quad
  \| \bq_h \|_{\bX_h} = \max_{T \in \calT_h } \max_{j=\overline{1,\calN_T} }
  \left\{ |\bq_h(z_{j,T})| \right\},
\]
and the discrete inner product $\scl\cdot,\cdot\scr_h$ is defined by the quadrature rule
\[
  \scl \bq_h, \br_h \scr_h = \sum_{T\in\calT_h} |T|\sum_{j=1}^{\calN_T} \omega_j 
  \bq_h(z_{j,T}) \SCAL \br_h(z_{j,T}).
\]
The latter induces the norm $\| \bq_h \|_h^2 = \scl \bq_h , \bq_h \scr_h$,
which clearly implies
\begin{equation}
  \| \bq_h \|_h \leq c \| \bq_h \|_{\bL^2}, \quad \forall q_h \in \bX_h.
\label{eq:discip}
\end{equation}

In this setting it is not difficult to see that the fully discrete subgradient flow 
\eqref{eq:ourprobfulldiscvarineq} with
energy \eqref{eq:defofdiscEnergy} is equivalent to finding $u_h^\dt \subset \polV_h$ and
$\bz_h^\dt \subset \frakB_1(\bX_h)$ that solve:
\begin{equation}
\label{eq:fulldisc1}
\scl \frac{\frakd u_h^{k+1}}\dt, w_h \scr + \scl \bz_h^{k+1}, \GRAD w_h \scr_h
= 0, \quad \forall w_h \in \polV_h,
\end{equation}
and
\begin{equation}
  \scl \bq_h - \bz_h^{k+1}, \GRAD u_h^{k+1} \scr_h \leq 0, \quad \forall \bq_h \in \frakB_1(\bX_h).
\label{eq:fulldisc2}
\end{equation}

\subsection{Exact Solver}
\label{sub:ExSolver}
Problem \eqref{eq:fulldisc1}--\eqref{eq:fulldisc2} is not a practical numerical scheme,
since it involves the solution of the (local) variational inequality \eqref{eq:fulldisc2}. To 
overcome this difficulty we exploit that \eqref{eq:fulldisc1}--\eqref{eq:fulldisc2}
are the optimality conditions of the functional
$w_h \mapsto \tfrac1{2\dt} \| w_h - u_h^k \|_{L^2}^2 + \int_\Omega\calI_h |\GRAD w_h |$,
which motivates the following algorithm for the solution of 
\eqref{eq:fulldisc1}--\eqref{eq:fulldisc2}
(\cf \cite{BartelsTV,MR2782122}): Let $\sigma>0$ and $\tau>0$ be parameters to be chosen.
Given $g_h=u_h^k \in \polV_h$, set $v_h^0 = g_h$, $\blambda_h^0 = 0$
and find $\{ v_h^\tau, \blambda_h^\tau \} \subset \polV_h \times \frakB_1(\bX_h)$ by
\begin{equation}
  \scl -\frac\sigma\tau \frakd \blambda_h^{l+1} + \GRAD v_h^{\star,l+1},
  \bq_h - \blambda_h^{l+1} \scr_h \leq 0 \quad \forall \bq_h \in \frakB_1(\bX_h),
\label{eq:itlambda}
\end{equation}
where $v_h^{\star,l+1}$ is the extrapolation defined in \eqref{eq:defofstar}, and
\begin{equation}
  \scl \frac{\frakd v_h^{l+1}}\tau, w_h \scr + \scl \blambda_h^{l+1}, \GRAD w_h \scr_h
  + \frac1\dt \scl v_h^{l+1}, w_h \scr = \frac1\dt \scl g_h, w_h \scr.
\label{eq:itv}
\end{equation}
Finally, set $u_h^{k+1}=v_h^\infty$.

The last assertion needs to be justified, this is given in the following.

\begin{theorem}[Convergence of the inner iterations]
\label{thm:convinner}
Assume the inverse inequality $\| \GRAD w_h \|_{L^2} \leq c_i h^{-1} \| w_h \|_{L^2}$
holds, and the parameters $\tau$ and $\sigma$ are chosen so that $\tau \leq c_i \sigma h$.
Then the inner iterative loop \eqref{eq:itlambda}--\eqref{eq:itv} converges, in the
sense that for every $L\geq1$,
\[
  \sum_{l=0}^L \left(
    \frac\tau\dt \| u_h^{k+1} - v_h^{l+1} \|_{L^2}^2 
    + \frac\sigma2 \| \frakd \blambda_h^{l+1} \|_h^2
  \right) \leq c.
\]
\end{theorem}
\begin{proof}
We follow the arguments of \cite[Proposition 3.1]{BartelsTV}. To alleviate the notation, set
$e_h^l = u_h^{k+1} - v_h^l$ and $\bE_h^l = \bz_h^{k+1} - \blambda_h^l$.
Subtract \eqref{eq:itv} from \eqref{eq:fulldisc1} to obtain
\begin{equation}
  \scl \frac{\frakd e_h^{l+1} }\tau, w_h \scr + \scl \bE_h^{l+1}, \GRAD w_h \scr_h
  + \frac1\dt \scl e_h^{l+1}, w_h \scr
  = 0.
\label{eq:errv}
\end{equation}

Set $\bq_h = \bz_h^{k+1}$ in \eqref{eq:itlambda}, $\bq_h = \blambda_h^{l+1}$ in
\eqref{eq:fulldisc2} and add them, to obtain
\begin{equation}
  -\scl -\frac\sigma\tau \frakd \bE_h^{l+1} 
  + \GRAD\left( u_h^{k+1} - v_h^{\star,l+1} \right), \bE_h^{l+1} \scr_h \leq 0.
\label{eq:errlambda}
\end{equation}

Set $w_h = 2\tau e_h^{l+1}$ in \eqref{eq:errv} to obtain
\begin{equation}
  \frakd \| e_h^{l+1} \|_{L^2}^2 + \| \frakd e_h^{l+1} \|_{L^2}^2
  + 2\tau \scl \bE_h^{l+1}, \GRAD e_h^{l+1} \scr_h + \frac{2\tau}\dt \| e_h^{l+1} \|_{L^2}^2
  \leq 0.
\label{eq:sanspLap}
\end{equation}
Multiply \eqref{eq:errlambda} by $2\tau$ and add it to \eqref{eq:sanspLap} and add over
$l=\overline{0,L-1}$ to obtain
\begin{multline*}
  \| e_h^L \|_{L^2}^2 + \sigma \| \bE_h^L \|_h^2
  + \sum_{l=0}^{L-1} \left( 
    \| \frakd e_h^{l+1} \|_{L^2}^2
    + \sigma \| \frakd \bE_h^{l+1} \|_h^2
    + \frac{2\tau}\dt \| e_h^{l+1} \|_{L^2}^2
    \right) \\
  \leq \| e_h^0 \|_{L^2}^2 + \sigma \| \bE_h^0 \|_h^2
  + 2\tau \sum_{l=0}^{L-1} \scl  \frakd^2 \GRAD e_h^{l+1}, \bE_h^{l+1} \scr_h.
\end{multline*}
Now we proceed as in \cite[Proposition 3.1]{BartelsTV} with the last term: we use \eqref{eq:abel} 
to sum by parts and next employ
inequality \eqref{eq:discip}, the inverse inequality and
repeated applications of Cauchy-Schwarz to obtain the desired result.
\end{proof}

\begin{remark}[Convergence of the Lagrange multiplier]
As already mentioned in \cite[Remark 3.2(i)]{BartelsTV}, convergence 
$\blambda_h^l \rightarrow \bz_h^{k+1}$, as $l\rightarrow \infty$, cannot be expected in general, due to 
the non-uniqueness of $\bz_h^{k+1}$. However the increments
$\frakd \blambda_h^l$ converge to zero.
\end{remark}

\begin{remark}[Solution of the discrete variational inequality]
It is not difficult to see that the solution to the discrete variational inequality \eqref{eq:itlambda}
is explicit
\[
  \blambda_h^{l+1}(z_{j,T}) = \dfrac{ \blambda_h^l + \tfrac\sigma\tau \GRAD v_h^{\star,l+1} }
  { \max\{ 1, | \blambda_h^l + \tfrac\sigma\tau \GRAD v_h^{\star,l+1} |(z_{j,T}) \} }.
\]
\end{remark}

\subsection{Inexact Solver}
\label{sub:InexSolver}
It is not feasible to iterate in \eqref{eq:itlambda}--\eqref{eq:itv} until convergence.
For this reason we included in the analysis presented in \S\ref{sub:EulerandFEM}
a perturbation term since if we were to stop after $L-1$ iterations, 
\eqref{eq:itv} would become
\[
  \frac1\dt \scl v_h^L - u_h^k, w_h \scr + \scl \blambda_h^L, \GRAD w_h \scr_h
  = -\frac1\tau \scl \frakd v_h^L, w_h \scr,
\]
and \eqref{eq:itlambda} would read
\[
  \scl \br_h^L + \GRAD v_h^L, \bq_h - \blambda_h^L \scr_h \leq 0,
\]
with $\br_h^L = -\tfrac\sigma\tau \frakd \blambda_h^L - \frakd^2 \GRAD v_h^L$.
Doing the formal replacements $v_h^L \leftarrow \widetilde u_h^{k+1}$ and 
$\blambda_h^L \leftarrow \widetilde\bz_h^{k+1}$
on the left hand side of these identities, setting $w_h = \widetilde u_h^{k+1} - \widetilde w_h$ and adding them
we obtain
\begin{equation}
  \scl \frac{\frakd \widetilde u_h^{k+1} }\dt, \widetilde u_h^{k+1} - \widetilde  w_h \scr +
  \widetilde\Psi_h(\widetilde u_h^{k+1}) - \widetilde\Psi_h(\widetilde  w_h) 
  \leq \scl -\frac{ \frakd v_h^L }\tau, \widetilde u_h^{k+1} - \widetilde w_h \scr,
\label{eq:fulldiscvarineqperturbed}  
\end{equation}
with
\[
  \widetilde\Psi_h(w_h) =  \sum_{T \in \calT_h }\int_T \calI_h |\GRAD w_h + \br_h^L|.
\]
The non-homogeneous term on the right hand side can be understood as a perturbation
$\rho^\dt$. In addition, the modified discrete energies
can be controlled owing to Theorem~\ref{thm:convinner}. We make these ideas rigorous in the following.

\begin{theorem}[Convergence of fully discrete scheme with inexact solutions]
Let $\Omega$ be star-shaped with respect to a point and assume that $\ue_0 \in \BV \cap L^\infty(\Omega)$.
If $\{ \widetilde u_h^\dt, \widetilde \bz_h^\dt \} \subset \polV_h \times \bX_h$ are approximations to
the solution of \eqref{eq:fulldisc1}--\eqref{eq:fulldisc2}, computed with algorithm
\eqref{eq:itlambda}--\eqref{eq:itv} in such a way that, for every time step $k$, the inequalities
\begin{equation}
  \| \tau^{-1} \frakd v_h^L \|_{L^2} \leq c \dt^{1/2}, \qquad
  \| \br_h^L \|_{L^1} \leq c \dt,
\label{eq:supercondition}
\end{equation}
are satisfied, then the following error estimate holds
\[
  \left\| \ue - \widehat{ \widetilde u_h^{k+1}} \right\|_{L^\infty(L^2)} \leq 
  \| \ue_0 - u_h^0 \|_{L^2} +
  c( \dt^{1/2} + h^{1/6} ).
\]
The spatial rate of convergence becomes $\calO(h^{1/4})$ as soon as $\Omega = (0,1)^d$ and the
mesh $\calT_h$ is Cartesian.
\end{theorem}
\begin{proof}
Since $\ue_0 \in \BV\cap L^\infty(\Omega)$, the solution, $u^\dt$, to \eqref{eq:semidiscvarineq} 
converges, with order $\calO(\dt^{1/2})$, to the exact solution.
This also guarantees, via Corollary~\ref{cor:aprioribetter},
the convergence of $u_h^\dt$, solution of \eqref{eq:ourprobfulldiscvarineq}, to $u^\dt$
with order $\calO(h^{1/6})$.

To study the convergence of $\widetilde u_h^\dt$ to $u_h^\dt$ we proceed as in Theorem~\ref{thm:aprioribetter}.
Set $w_h = \widetilde u_h^{k+1}$ in \eqref{eq:ourprobfulldiscvarineq} and $\widetilde  w_h = u_h^{k+1}$ in
\eqref{eq:fulldiscvarineqperturbed} and add the result. On denoting $e_h^k = u_h^k - \widetilde u_h^k$,
we obtain 
\[
  \frac1\dt \scl \frakd e_h^{k+1}, e_h^{k+1} \scr + \Psi_h(u_h^{k+1}) - \widetilde \Psi_h(u_h^{k+1})
  + \widetilde \Psi_h(\widetilde u_h^{k+1}) - \Psi_h(\widetilde u_h^{k+1}) \leq
  \scl \frac{ \frakd v_h^L }{\tau}, e_h^{k+1} \scr.
\]
Since
\begin{align*}
  \widetilde \Psi_h(w_h) - \Psi_h(w_h) &=
  \sum_{T \in \calT_h }\int_T (\calI_h|\GRAD w_h + \br_h^L | - \calI_h|\GRAD w_h|)
  \leq \sum_{T \in \calT_h } \int_T \calI_h |\br_h^L | \\
  &\leq c \| \br_h^L \|_{L^1},
\end{align*}
we notice that condition \eqref{eq:supercondition} suffices to guarantee that
$\| e_h^\dt \|_{\ell^\infty(L^2)} \leq c \dt^{1/2}$.
Conclude with a trivial application of the triangle inequality.

The improvement on the spatial rate of convergence is due to Remark~\ref{rem:tvd}.
\end{proof}

\section{Numerical Experiments}
\label{sec:Numerics}

To illustrate the theory developed in the preceding sections, here we present a series of
numerical experiments. We implemented scheme \eqref{eq:itlambda}--\eqref{eq:itv} where the
stopping criterion for the inner iterations is given by \eqref{eq:supercondition}.
The implementation was done with
the help of the \texttt{deal.II} library (see \cite{BHK,BHK07}). Unless noted otherwise,
we set $\dt = \tfrac{\sqrt2 h}{10}$, $\tau = \dt$ and $\sigma = 0.1$.

\subsection{The Characteristic of a Convex Set}
\label{sub:Ball}
Let $E \Subset \Omega \subset \Real^2$, where $E$ is convex, connected and of 
finite perimeter. Define $ \lambda_E = \tfrac{ P(E) }{|E|}$,
where $P(E)$ stands for the perimeter of $E$.
Assume, in addition, that $\partial E \in \calC^{1,1}$ and that
the curvature of $E$, $\varkappa$ satisfies $\| \varkappa \|_{L^\infty(\partial E)} \leq \lambda_E$.
Then, according to \cite{MR2033382,MR1929886},
if $\ue_0 = \chi_E$ the solution to \eqref{eq:subTVflow} with homogeneous Dirichlet boundary conditions is
$\ue(x,t) = \left( 1 - \lambda_E t \right)^+ \chi_E(x)$.

\begin{figure}[ht]
  \begin{center}
    \includegraphics[scale=0.5]{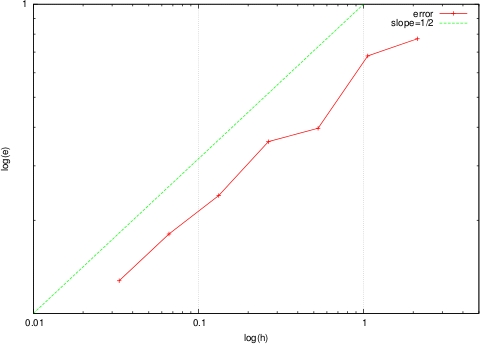}
  \end{center}
  \caption{$L^\infty(L^2)$-errors for the evolution of the characteristic $\chi_B$
  of a circle $B$ with exact solution 
  $\ue(x,t) = \left( 1 - \lambda_B t \right)^+ \chi_B(x)$ and $\lambda_B = P(B)/|B|$,
  \cite{MR2033382,MR1929886} (see \S\ref{sub:Ball}).}
\label{fig:charball}
\end{figure}

Figure~\ref{fig:charball} shows the $L^\infty(L^2)$-error between the discrete and numerical 
solutions in the case when $\Omega = (-3,3)^2$ and $E = B(0,1)$. As we can see the behavior of
the error is actually better (\ie $\calO(h^{1/2})$) than what our theory predicts.

\subsection{The Characteristic of a Disconnected Set}
\label{sub:ThreeBalls}
With the setting of \S\ref{sub:Ball}, references \cite{MR2033382,MR1929886} also show that if
$\ue_0 = \sum_{i=1}^3 \chi_{B(x_i,r_i)},$
where the centers are at the vertices of an equilateral triangle of unit side-length and the radii
satisfy $r_i \leq 0.2$, then 
$\ue(x,t) = \sum_{i=1}^3 (1-\lambda_{B_i}t)^+\chi_{B(x_i,r_i)}$.

\begin{figure}[ht]
\begin{center}
  \includegraphics[scale=0.15]{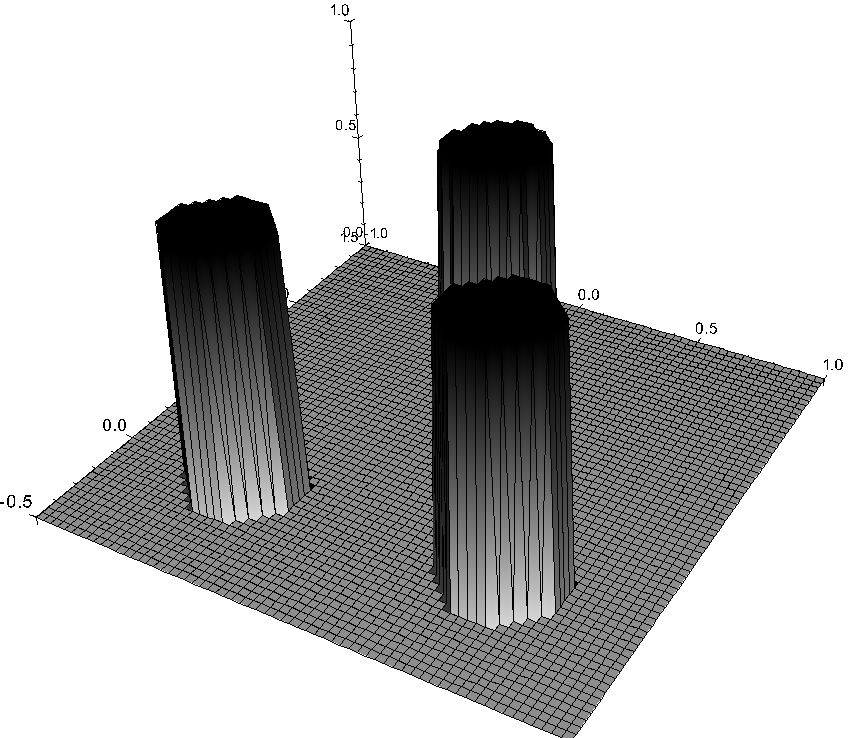}\hfil
  \includegraphics[scale=0.15]{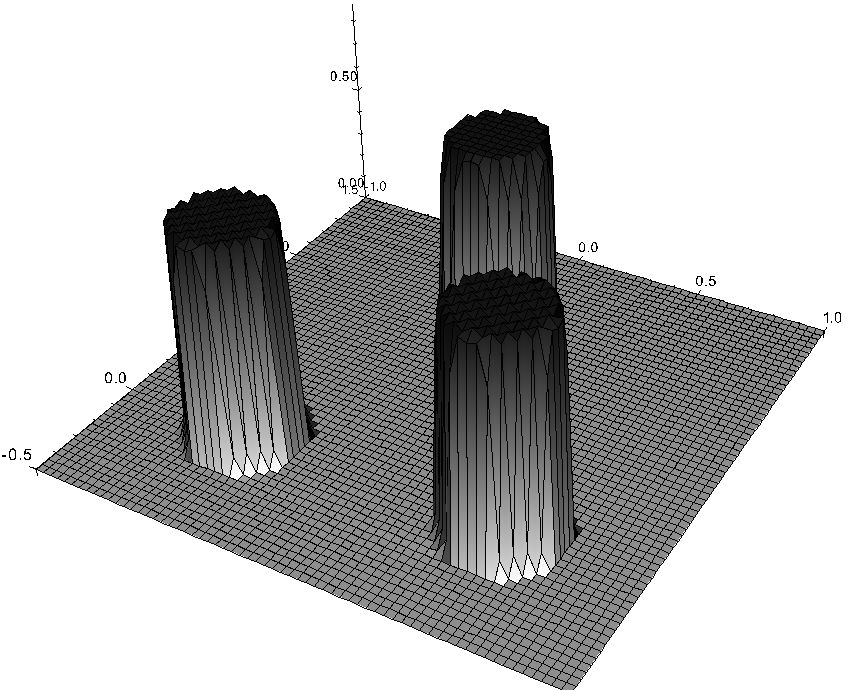} \\
  \includegraphics[scale=0.15]{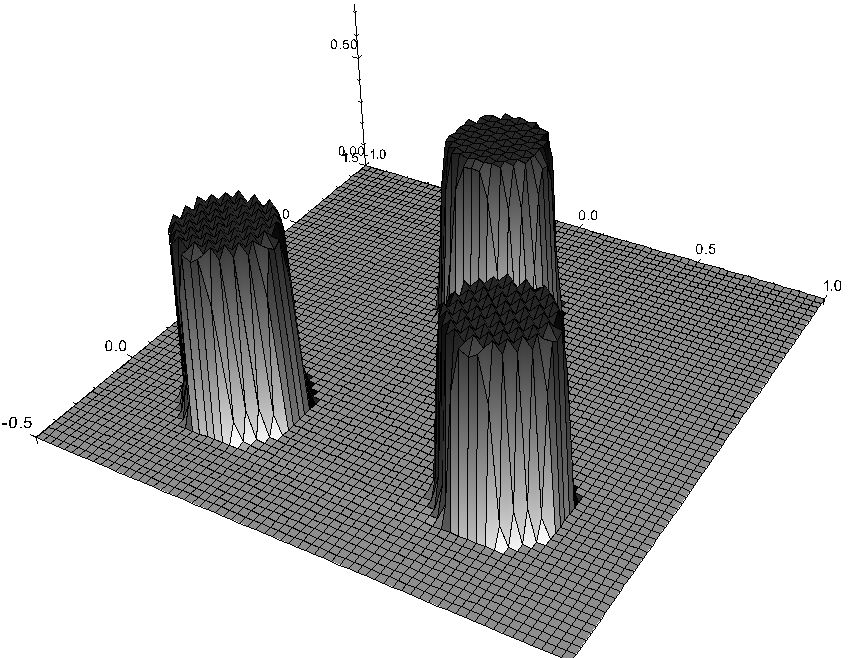}\hfil
  \includegraphics[scale=0.15]{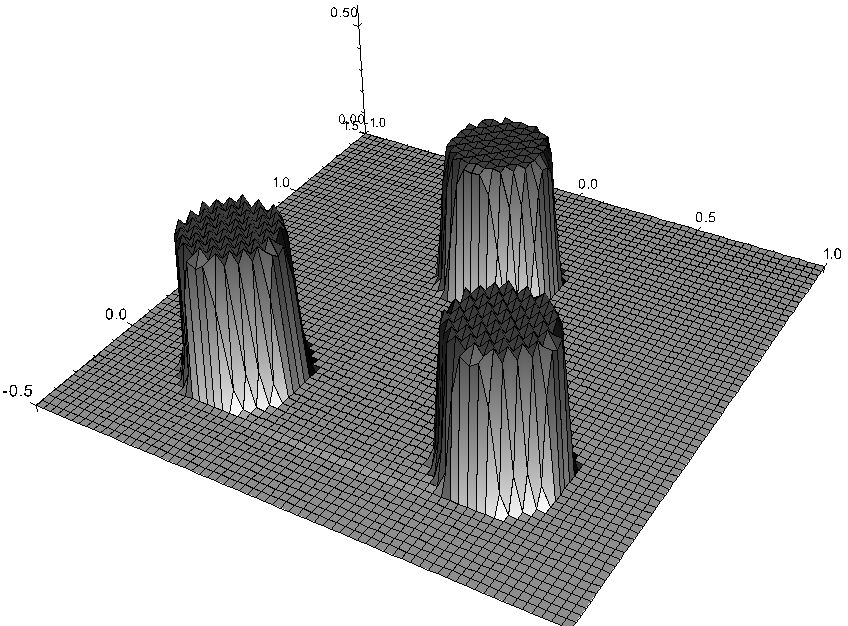} \\
  \includegraphics[scale=0.15]{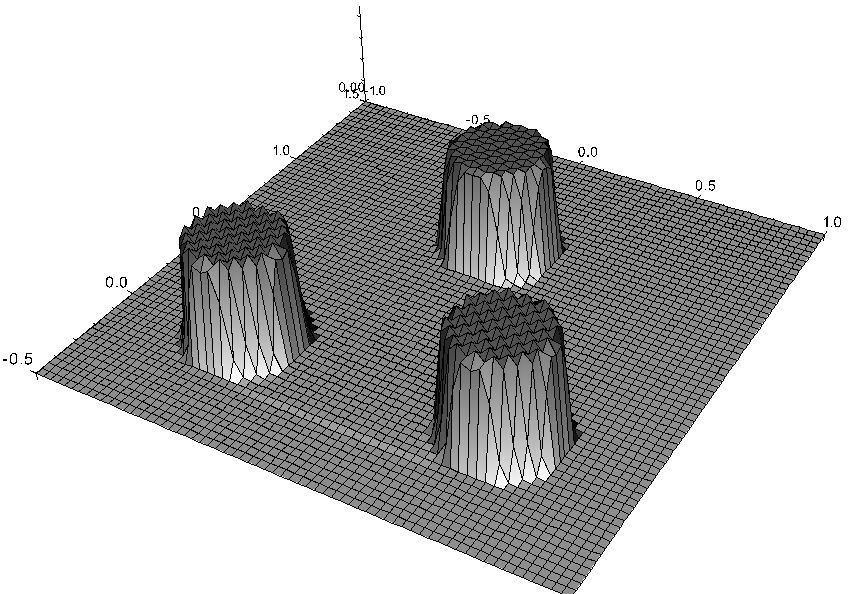}\hfil
  \includegraphics[scale=0.15]{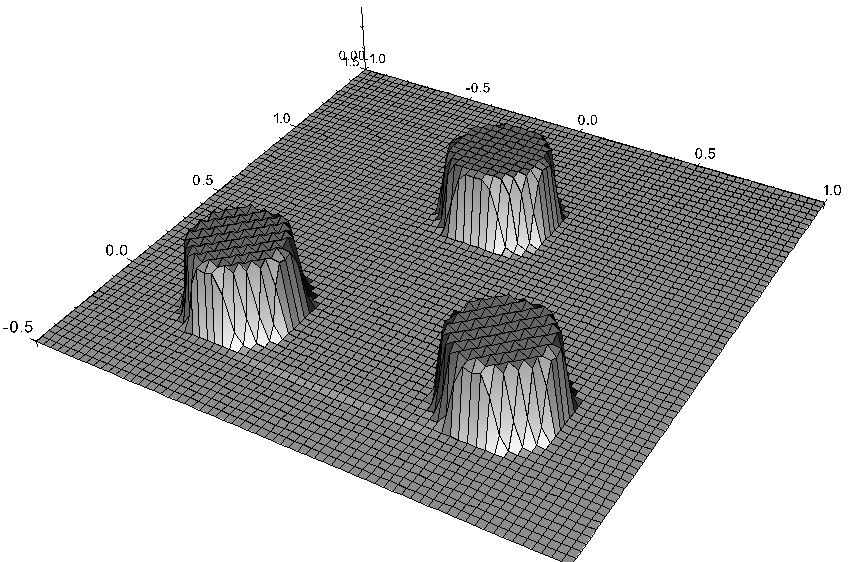} \\
  \includegraphics[scale=0.15]{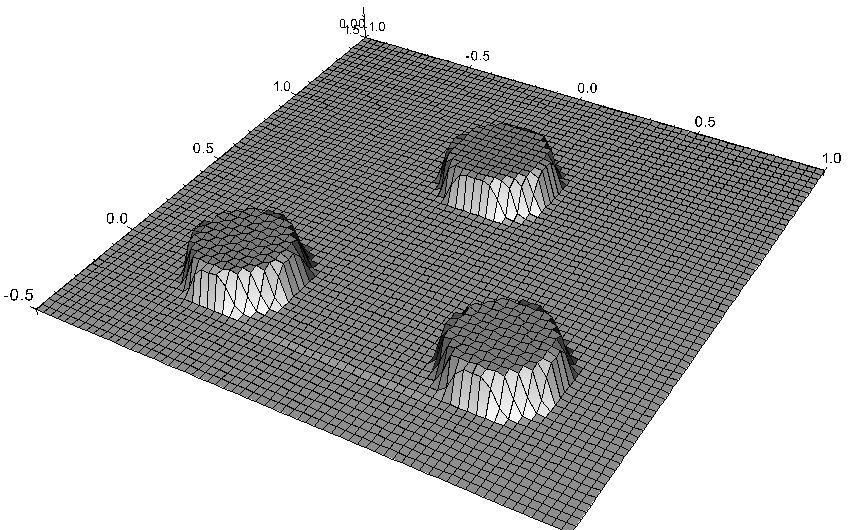}\hfil
  \includegraphics[scale=0.15]{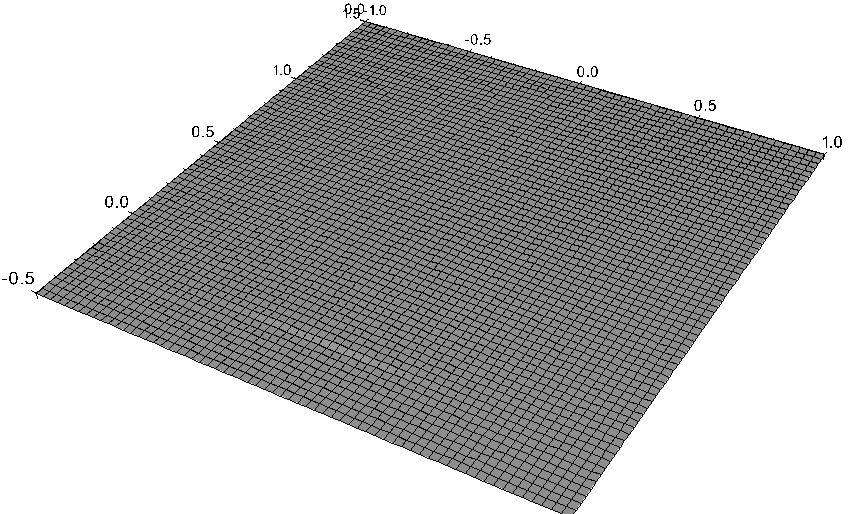}
\end{center}
\caption{Evolution of the characteristic set of three balls $B_i =B(x_i,r_i)$ with centers
$x_i$ at the vertices of an equilateral triangle of unit size and radii $r_i\leq0.2$. The exact
solution is $\ue(x,t) = \sum_{i=1}^3 (1-\lambda_{B_i}t)^+\chi_{B(x_i,r_i)}$.
The mesh is uniform with
size $h=2^{-5}$ and the time-step $\dt = \sqrt2 h/10$. The discrete solution is shown every three
time steps until extinction
and preserves the structure of $\ue$ without numerical diffusion and/or oscillations.}
\label{fig:threeballs}
\end{figure}

Figure~\ref{fig:threeballs} shows the evolution in this case when $h=0.03$. Notice that 
our method provides a good approximation with relatively few nodes and it 
does not introduce spurious oscillations. In addition, we obtain fairly good agreement with the 
extinction time, given by $t_\star = \tfrac1{\lambda_B} = \tfrac{r}2 = \tfrac{1}{10}.$
In contrast to methods that involve regularization \cite{MR1994316,MR2194526},
we observe consistency with the PDE
in that the support of the fully discrete solution remains constant in time. Finally,
the $L^\infty(L^2)$-norm of the error with respect to $h$ is shown in Figure~\ref{fig:threeballserr}.
Again, we see that the error is $\calO(h^{1/2})$.

\begin{figure}[ht]
  \begin{center}
    \includegraphics[scale=0.5]{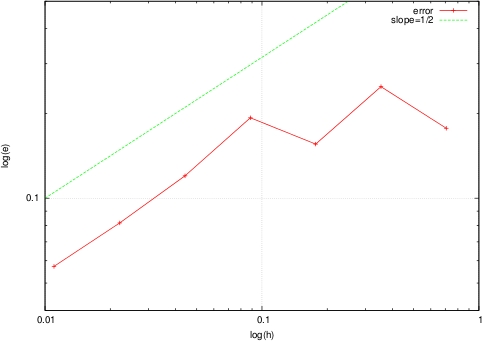}
  \end{center}
  \caption{$L^\infty(L^2)$-errors for the evolution of the characteristic of 
  three circles (see \S\ref{sub:ThreeBalls}). The observed convergence rate
  $\calO(h^{1/2})$ is better that predicted by the theory.}
\label{fig:threeballserr}
\end{figure}

\subsection{The Characteristic of an Annulus}
\label{sub:ring}
Again in the setting of \S\ref{sub:Ball}, the evolution of
$\ue_0 = M\chi_E$, where $E=B(0,R)\setminus \overline{B(0,r)}$ is an annulus and
$M,R,r\in\Real$, $r<R$, is governed by \cite{MR2033382,MR1929886}
\[
  \ue(x,t) = 
  \begin{dcases}
    \signum(M) ( |M| - \lambda_E t )^+ \chi_E + \lambda_{B(0,r)}t\chi_{B(0,r)}, &
      t < T_1, \\
    \signum(m) \left( |m| - \lambda_{B(0,r)}(t-T_1) \right)^+ \chi_{B(0,R)}, &
      t \geq T_1,
  \end{dcases}
\]
with
$T_1 = \tfrac{|M|}{ \lambda_E + \lambda_{B(0,r)} }$ and $m = \lambda_{B(0,r)}T_1$.

\begin{figure}[ht]
\begin{center}
  \includegraphics[scale=0.2]{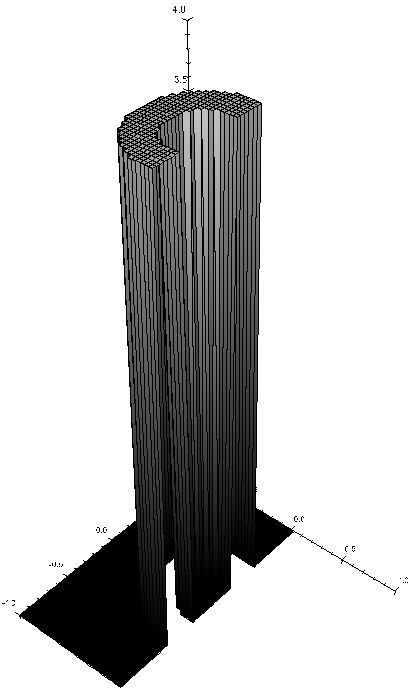}\hfil
  \includegraphics[scale=0.2]{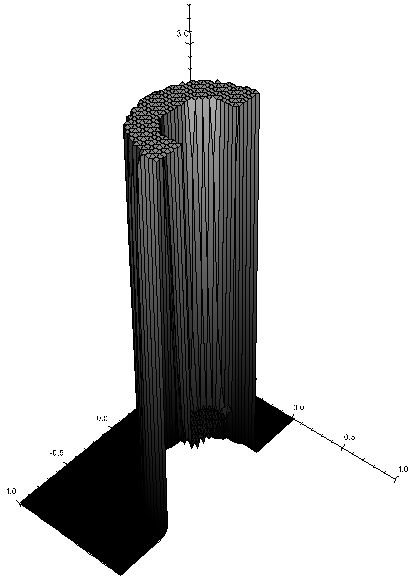}\hfil
  \includegraphics[scale=0.2]{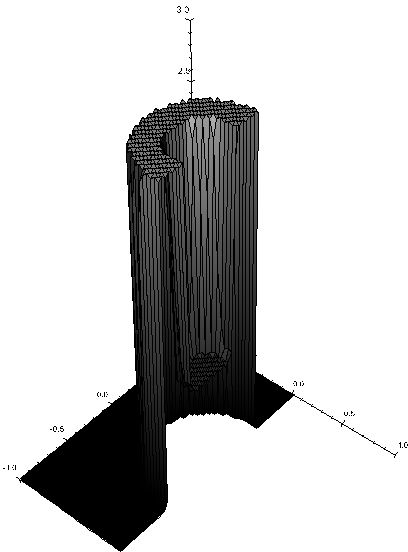}\\
  \includegraphics[scale=0.2]{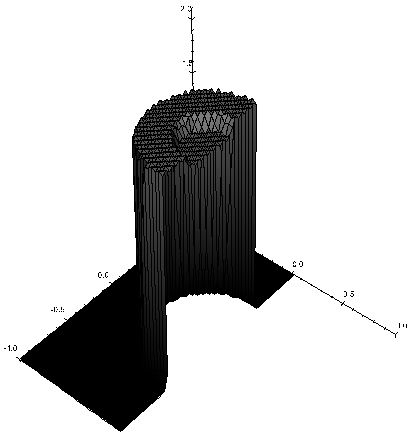}\hfil
  \includegraphics[scale=0.2]{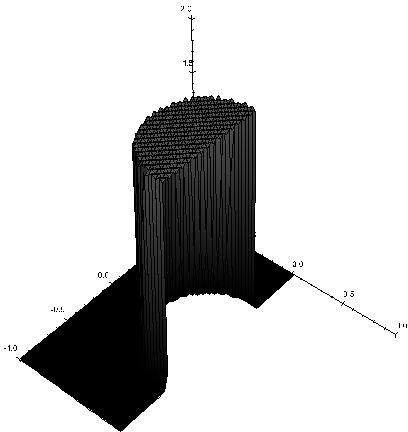}\hfil
  \includegraphics[scale=0.2]{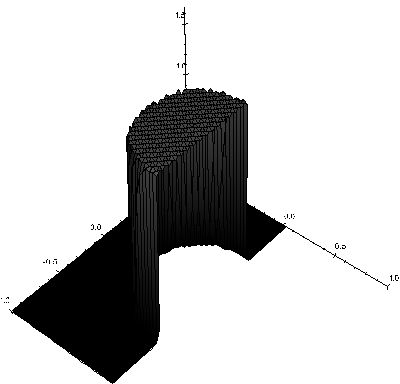}\\
  \includegraphics[scale=0.2]{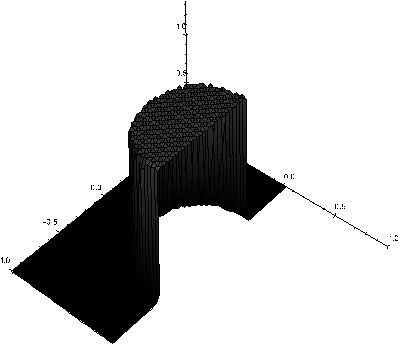}\hfil
  \includegraphics[scale=0.2]{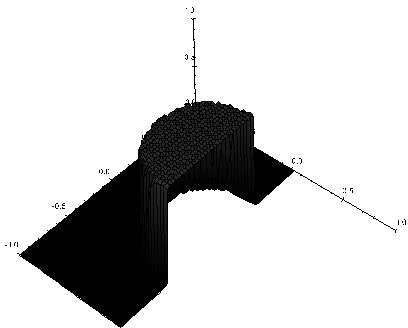}\hfil
  \includegraphics[scale=0.2]{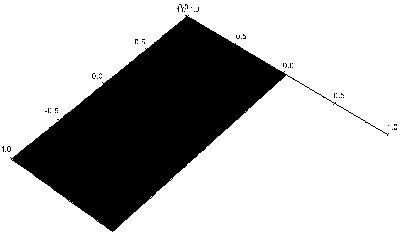}
\end{center}
\caption{Evolution of the characteristic of a ring
$E=B(0,R)\setminus \overline{B(0,r)}$ with $R=1/2$ and $r=1/4$.
The mesh is uniform with size $h=2^{-5}$ and the time-step $\dt = \sqrt2 h /10$.
The discrete solution, shown after 0, 18, 26, 50, 52, 72, 90, 108 and 170 time steps,
reproduces well the form of the exact solution without numerical diffusion. The 
sealed characteristic $\chi_{B(0,R)}$ decreases whereas $\chi_{B(0,r)}$ increases until
their heights merge and $\chi_{B(0,R)}$ continues to decrease.}
\label{fig:ring}
\end{figure}

\begin{figure}[ht]
\begin{center}
  \includegraphics[scale=0.5]{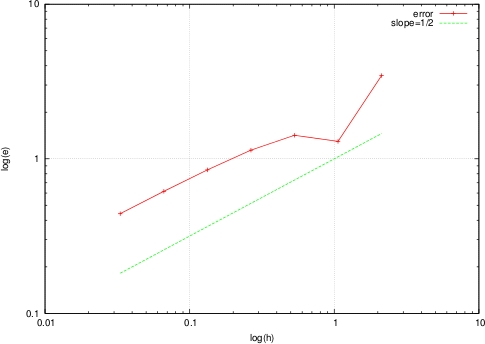}
\end{center}
\caption{$L^\infty(L^2)$-errors for the evolution of the characteristic of a ring
(see \S\ref{sub:ring}). The experimental convergence rate $\calO(h^{1/2})$ is better than
predicted.}
\label{fig:ring_order}
\end{figure}

We set $M=4$, $R=1/2$, $r=1/4$.
Figure~\ref{fig:ring} shows the evolution of the numerical solution for $h=2^{-5}$ and
notice that, in this case, $T_1 = 1/4$, $m = 2$ and the extinction time is
$T_{ext} = 3/4$. Numerically we obtained $T_{1,h} \approx 0.2558$, $m_h \approx 2.045$,
$T_{ext,h} \approx 0.796$; which are in good agreement with the exact
values. In addition, we see that there is no diffusion, as opposed to what methods based
on regularization obtain.
The $L^\infty(L^2)$-error with respect to $h$ is presented in 
Figure~\ref{fig:ring_order}. Again we obtain $\calO(h^{1/2})$.

\subsection{Convergence of the Inner Loop}
\label{sub:ConvInner}

\begin{figure}[ht]
\begin{center}
  \includegraphics[scale=0.5]{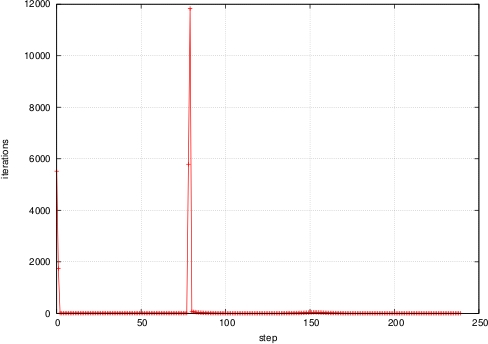}\hfil
  \includegraphics[scale=0.5]{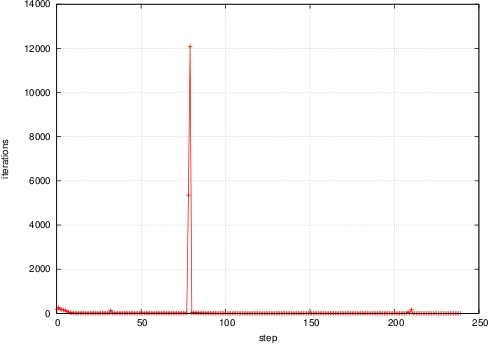}
\end{center}
\caption{Number of iterations.
Left: the initial dual variable is not in the subdifferential of the initial data,
which requires more than $5000$ initial iterations for convergence.
Right: the initial dual variable is in the subdifferential,
which entails fewer initial iterations for convergence (see \S\ref{sub:ConvInner}).}
\label{fig:its}
\end{figure}

The result of Theorem~\ref{thm:convinner} guarantees convergence of the inner iteration at 
every step. However, there is no assessment of the speed of convergence. We numerically
investigate the effect of the choice of initial condition for the inner loop.
Set $\Omega = (-2,2)$ and $\ue_0=\chi_{[-1,1]}$.
We choose
\[
  \bz_0^1 = 0,
  \quad
  \bz_0^2 = \begin{dcases}
              x + 2, & x \in (-2,-1), \\
              -x, & x \in (-1,1), \\
              x -2, & x \in (1,2).
            \end{dcases}
\]
It is possible to show that only $\bz_0^2 \in \partial\Psi(\ue_0)$.
For $h=2^{-6}$, $\dt = \tau = h/10$ and $\sigma=1$, Figure~\ref{fig:its} plots the number of
iterations as a function of the step for $\bz_0^1$ and $\bz_0^2$.
The number of initial iteration is much higher (more than $5000$) for $\bz_0^1$ than for
$\bz_0^2$. Moreover,
notice the spike on both graphs at 79 steps. This is due to the fact that, at this step, extinction
occurs. From this we conclude that
the number of iterations heavily depends on the initial choice of $\bz_0$ or, more generally,
on $\blambda_h^0$. Strategies for choosing it need further investigation.

\subsection{Comparison with Regularized Flow}
\label{sub:compreg}
To conclude our discussion, it is imperative to make a comparison between our method and those
that involve regularization and show that the advantages in our approach are numerous. 
Before embarking in such an endeavor, let us recall some properties of regularized flows. 
The analysis developed in \cite{MR2194526} does not provide a clear understanding
of the relation between the regularization $\epsilon$ and the space discretization $h$. By means
of numerical experiments the authors conclude that $\epsilon = \calO(h^2)$ is the optimal scaling 
law, see \cite[Figures 4--5]{MR2194526}.

Let us, with the help of the theory developed in \S\ref{sub:EulerandFEM}, try to bring some light
into this matter. If we denote by $\ue_\epsilon$ the solution to \eqref{eq:tvflowstrong} with 
regularization
$\Psi_\epsilon(w)=\int_\Omega \sqrt{\epsilon^2 + |\GRAD w|^2}$ given by
\eqref{eq:regularization}, \cite[Theorem 2]{MR2194526} shows that
\begin{equation}
\label{eq:uMueps}
  \| \ue - \ue_\epsilon \|_{L^\infty(L^2)} \leq c \epsilon^{1/2}.
\end{equation}
If $u_{\epsilon,h}^\dt \subset \polV_h$ is the solution to a fully discrete approximation 
of the regularized flow with $\Psi_{\epsilon,h} =\Psi_\epsilon$,
under the assumption that $\ue_\epsilon \in L^\infty(0,T;\Hdeux)$ we have
\begin{equation}
\label{eq:uepsMuepshdt}
  \| \ue_\epsilon - \widehat{u}_{\epsilon,h}^\dt \|_{L^\infty(L^2)}
  \leq c \left (\dt^{1/2} + h^{1/2}\epsilon^{-\alpha/2} \right).
\end{equation}
To see this it suffices to realize that \eqref{eq:discenergymonotone} is trivially satisfied and,
setting $\calC_h = \Pi_h$ we conclude
that \eqref{eq:interpolationabs} amounts to
\[
  \| \calI_h w - w \|_{L^2} \leq c h^2 \| w \|_{H^2}
  \quad \Longrightarrow \quad
  \vare_1(h) = ch^2 \| w \|_{H^2},
\]
and, since
\[
  \int_\Omega \left| \sqrt{ \epsilon^2 + |\GRAD w_1 |^2 } - \sqrt{ \epsilon^2 + |\GRAD w_2 |^2 } \right|
  \leq
  \int_\Omega \left| \GRAD (w_1 - w_2) \right|,
\]
we have, for \eqref{eq:interpolationabstracr},
\[
  \vare_2(h) \leq c h \| w \|_{W^2_1}.
\]
Finally, \cite[Theorem~1.2]{MR1994316} shows that $\| \ue_\epsilon \|_{L^2(H^2)} \leq \epsilon^{-\alpha}$
for some $\alpha \in \polN_0$. Combining \eqref{eq:uMueps} and \eqref{eq:uepsMuepshdt} 
with Theorem~\ref{thm:aprioribetter}
we conclude
\begin{equation}
\label{eq:uMuepshdt}
  \| \ue - \widehat{u}_{\epsilon,h}^\dt \|_{L^\infty(L^2)}
  \leq c \left( \dt^{1/2} + \epsilon^{1/2} + h\epsilon^{-\alpha/2} + h^{1/2}\right),
\end{equation}
which yields the optimal scaling $\epsilon = \calO( h^{\frac{2}{1+\alpha}})$.

The last ingredient we need to make the comparison is to
recall that (\cf \cite[Remark 5.3]{MR1301176}) in the one dimensional case, \ie $d=1$,
if the initial data is monotone, it is itself a minimizer of the total variation energy, and so
the flow fixes it. In other words, if $\ue_0$ is monotone, then $\ue(t)=\ue_0$ for all $t>0$.

\begin{figure}[ht]
\begin{center}
  \includegraphics[scale=0.5]{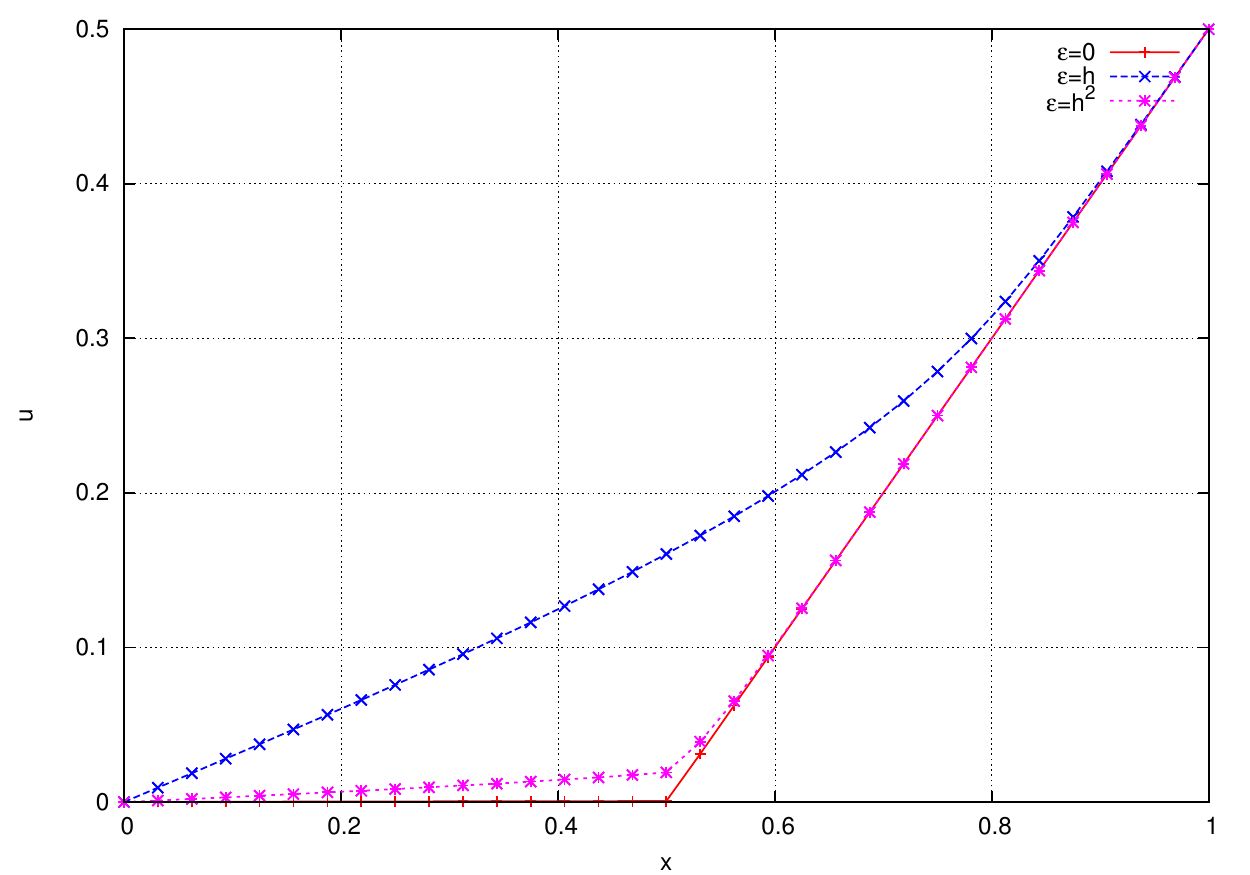}
\end{center}
\caption{Solution, at $T=5$, of the TV flow with monotone initial data.
Our method (red with $+$), regularized flow with $\epsilon = h$ (blue with $\times$) and 
with $\epsilon=h^2$ (magenta with $*$).
$h=2^{-5}$, $\dt = 2^{-10}$. The unregularized solution coincides with
the exact solution, whereas the regularized flow misses it, although it belongs to the 
discrete space.}
\label{fig:compreg}
\end{figure}

Consider, in $\Omega = (0,1)$, the initial data
\[
  \ue_0(x) = \begin{dcases}
               0, & x \leq \tfrac12, \\
               x - \tfrac12, & \tfrac12 < x \leq 1.
             \end{dcases}
\]
According to the discussion presented above, the solution to \eqref{eq:subTVflow}
is $\ue(t) = \ue_0$. Figure~\ref{fig:compreg} shows the solution, at $T=5$, obtained with our 
method and the regularized flow with $\epsilon = \calO(h^2)$ and $\epsilon=\calO(h)$. The first
choice is the one advocated in \cite{MR2194526}; while the second is the optimal according to
\eqref{eq:uMuepshdt} 
provided $\alpha=1$, which is consistent with Proposition~\ref{prop:approxbv}. In such a case,
\eqref{eq:uMuepshdt} gives the following error estimate
\[
  \| \ue - \widehat{u}_{\epsilon,h}^\dt \|_{L^\infty(L^2)} \leq c h^{1/2},
\]
provided $\epsilon \approx h \approx \dt$.
The mesh size is $h = 2^{-5}$ and the time-step $\dt = 2^{-10}$. Notice that the requirement
$\dt = \calO(h^2)$, which is needed for $L^2$-convergence in the regularized flow
(\cf \cite[Theorem~4]{MR2194526} and \cite[Theorem~1.7]{MR1994316}), is satisfied.

The advantages of our method are now evident. We do not impose any restriction on the time-step,
as opposed to the $\dt = \calO(h^2)$ that is necessary in regularized methods to guarantee convergence in
$\Ldeux$. Even if one is willing to settle for convergence in $L^p(\Omega)$, with $p<d/(d-1)$, the methods
with regularization require that the solution of the regularized flow is in
$L^\infty(0,T;W^1_1(\Omega)) \cap L^\infty(0,T;H^1_{\text{loc}}(\Omega))$. If an error estimate is desired,
one must impose that $\ue_0 \in \calC^2(\bar\Omega)$, and even in that case it is not clear
what is the relation between $h$, $\dt$ and $\epsilon$. In addition to these approximation issues,
the regularized flow requires the solution, at each time step, of a nonlinear system and no convergence
analysis is provided. In contrast, we have developed and analyzed an inexact 
iterative scheme for the solution
of our problems at each time step, and we have showed its global convergence. To finalize, the result
presented in Figure~\ref{fig:compreg} shows that the regularized flow misses certain fundamental features
of the problem, even in simple cases.

\section{Total Variation Minimization}
\label{sec:TVmin}
We conclude with yet another application of our result on approximation of functions 
of bounded variation (Proposition~\ref{prop:approxbv}): we improve on the existing results about
total variation minimization. Let $g\in L^\infty(\Omega)$ and $\alpha>0$. Consider
$\Xi(w) = \Psi(w) + \tfrac\alpha2 \| w - g \|_{L^2}^2$.
Thanks to the fact that this functional is strictly convex there exists a unique
 $\xi \in \BV\cap\Ldeux$ such that
\begin{equation}
  \Xi(\xi) = \inf\left\{ \Xi(w): w \in \Ldeux \right\} < \infty.
\label{eq:defofustar}
\end{equation}
Here we are interested in the approximation of $\xi$ by elements of $\polV_h$.
Since $\polV_h$ is finite dimensional, there is a unique $\xi_h \in \polV_h$ such that
\begin{equation}
  \Xi(\xi_h) = \inf\left\{ \Xi(w_h): w \in \polV_h \right\}  < \infty.
\label{eq:defofuhstar}
\end{equation}
The main approximation properties of $\xi_h$ are detailed in the following.

\begin{theorem}[Convergence of discrete minimizers]
\label{thm:convminimzer}
Assume that $\Omega$ is star shaped with respect to a point. Let
$\xi \in \BV\cap L^\infty(\Omega)$ and $\xi_h \in \polV_h$ be defined as in \eqref{eq:defofustar} and
\eqref{eq:defofuhstar}, respectively. Then
\[
  \| \xi - \xi_h \|_{L^2} \leq c h^{1/4}.
\]
\end{theorem}
\begin{proof}
We adapt the ideas presented in \cite[Theorem 3.1]{BartelsTV}.
Let $\epsilon>0$ and $\xi_\epsilon \in \calC^\infty(\Omega)$ be an approximation of
$\xi$ that satisfies all the properties stated in Proposition~\ref{prop:approxbv}.
Owing to the strict convexity of $\Xi$ and the fact that $\xi_h$ is a discrete minimizer, we have
\begin{align*}
  \frac\alpha2 \| \xi - \xi_h \|_{L^2}^2 &\leq \Xi( \xi_h) - \Xi( \xi )
      \leq \Xi(\Pi_h \xi_\epsilon) - \Xi(\xi) \\
      &= \left( \| \GRAD \Pi_h \xi_\epsilon \|_{L^1} - | \De \xi |(\Omega) \right)
      + \frac\alpha2 \left( \| \Pi_h \xi_\epsilon - g \|_{L^2}^2 - \| \xi - g \|_{L^2}^2 \right)
      = \calA_1 + \calA_2,
\end{align*}
where $\Pi_h$ is the Cl\'ement interpolation operator \cite{MR0400739}.
Let us look at each one of the terms in this last inequality:
\begin{enumerate}[$\calA_1$:]
  \item We add and subtract the $W^1_1$-seminorm of $\xi_\epsilon$ and use its approximation properties
  with respect to total variation along with the approximation properties of $\Pi_h$ described in
  \eqref{eq:lagrange}:
  \[
      \calA_1 \leq \| \GRAD (\Pi_h \xi_\epsilon - \xi_\epsilon ) \|_{L^1} 
      + \| \GRAD \xi_\epsilon \|_{L^1} - |\De \xi|(\Omega)
      \leq c \left( \frac{h}\epsilon + \epsilon \right) | \De \xi|(\Omega).
  \]
  \item We, again, use the approximation properties of $\xi_\epsilon$ and the fact that
  $\xi$ and $g$ are essentially bounded, say by a constant $c>0$,
  \[
    \calA_2 \leq \frac{c\alpha}2 \left( \| \calI_h \xi_\epsilon - \xi_\epsilon \|_{L^1}
      + \| \xi_\epsilon - \xi \|_{L^1} \right)
      \leq c \left( \frac{h^2}{\epsilon} + \epsilon \right) | \De \xi |(\Omega).
  \]
\end{enumerate}
Setting $\epsilon = h^\frac12$ we obtain the result.
\end{proof}

\begin{remark}[Convergence of total variation minimization]
\label{rem:convergenceBesov}
Theorem~\ref{thm:convminimzer} is, in a sense, an improvement over the original result of \cite{BartelsTV},
at least for star shaped domains, and under the boundedness assumptions on 
$g$ and $\xi$. If $\xi \in B^s_\infty(\Ldeux)$ for some $s \in (0,1]$, and 
relying on the results of \cite{MR2792398}, \cite[Theorem~3.1]{BartelsTV} proves the estimate
\[
  \| \xi - \xi_h \|_{L^2} \leq c h^{\frac{s}{2(1+s)}},
\]
so that the best possible rate of convergence is $\calO(h^{1/4})$, which is what we
obtain, but with lower regularity. To understand this regularity assumption
it suffices to recall that
$\BV\cap L^\infty(\Omega) \hookrightarrow B^s_\infty(L^2(\Omega))$ for $s\leq \tfrac12$
(see \cite[Lemma~38.1]{MR2328004} for a proof and, in some sense, the converse inclusion).
The key step that allowed us to reduce the smoothness assumption is
Proposition~\ref{prop:approxbv}. In addition, the proof of Theorem~\ref{thm:convminimzer}
shows that if we had a TV-diminishing interpolant, we would obtain 
\[
  \| \xi - \xi_h \|_{L^2} \leq c h^{1/2},
\]
which is an optimal error estimate for $\xi \in B^s_\infty(\Ldeux)$.
Such a construction is presented in \cite{TVDInterpolation}.
\end{remark}
 
\bibliographystyle{siam}
\bibliography{biblio}

\def\cprime{$'$}
\begin{thebibliography}{10}

\bibitem{MR1857292}
{\sc L.~Ambrosio, N.~Fusco, and D.~Pallara}, {\em Functions of bounded
  variation and free discontinuity problems}, Oxford Mathematical Monographs,
  The Clarendon Press Oxford University Press, New York, 2000.

\bibitem{MR2033382}
{\sc F.~Andreu-Vaillo, V.~Caselles, and J.M. Maz{\'o}n}, {\em Parabolic
  quasilinear equations minimizing linear growth functionals}, vol.~223 of
  Progress in Mathematics, Birkh\"auser Verlag, Basel, 2004.

\bibitem{MR750538}
{\sc G.~Anzellotti}, {\em Pairings between measures and bounded functions and
  compensated compactness}, Ann. Mat. Pura Appl. (4), 135 (1983), pp.~293--318
  (1984).

\bibitem{BHK}
{\sc W.~Bangerth, R.~Hartmann, and G.~Kanschat}, {\em {\tt deal.{I}{I}}
  Differential Equations Analysis Library, Technical Reference}.
\newblock \texttt{http://www.dealii.org}.

\bibitem{BHK07}
\leavevmode\vrule height 2pt depth -1.6pt width 23pt, {\em deal.{II} --- a
  general-purpose object-oriented finite element library}, ACM Trans. Math.
  Softw., 33 (2007).

\bibitem{MR2582280}
{\sc V.~Barbu}, {\em Nonlinear differential equations of monotone types in
  {B}anach spaces}, Springer Monographs in Mathematics, Springer, New York,
  2010.

\bibitem{BartelsTV}
{\sc S.~Bartels}, {\em Total variation minimization with finite elements:
  Convergence and iterative solution}, SIAM Journal on Numerical Analysis, 50
  (2012), pp.~1162--1180.

\bibitem{MR1929886}
{\sc G.~Bellettini, V.~Caselles, and M.~Novaga}, {\em The total variation flow
  in {$\mathbb{R}^N$}}, J. Differential Equations, 184 (2002), pp.~475--525.

\bibitem{MR2197709}
{\sc M.~Breu{\ss}, T.~Brox, A.~B{\"u}rgel, T.~Sonar, and J.~Weickert}, {\em
  Numerical aspects of {TV} flow}, Numer. Algorithms, 41 (2006), pp.~79--101.

\bibitem{MR0348562}
{\sc H.~Br{\'e}zis}, {\em Op\'erateurs maximaux monotones et semi-groupes de
  contractions dans les espaces de {H}ilbert}, North-Holland Publishing Co.,
  Amsterdam, 1973.
\newblock North-Holland Mathematics Studies, No. 5. Notas de Matem{\'a}tica
  (50).

\bibitem{MR2338487}
{\sc M.~Burger, K.~Frick, S.~Osher, and O.~Scherzer}, {\em Inverse total
  variation flow}, Multiscale Model. Simul., 6 (2007), pp.~365--395
  (electronic).

\bibitem{MR2782122}
{\sc A.~Chambolle and T.~Pock}, {\em A first-order primal-dual algorithm for
  convex problems with applications to imaging}, J. Math. Imaging Vision, 40
  (2011), pp.~120--145.

\bibitem{MR2790462}
{\sc R.H. Chan, Y.~Dong, and M.~Hinterm{\"u}ller}, {\em An efficient two-phase
  {${\rm L}^1$}-{TV} method for restoring blurred images with impulse noise},
  IEEE Trans. Image Process., 19 (2010), pp.~1731--1739.

\bibitem{Ci78}
{\sc P.G. Ciarlet}, {\em The Finite Element Method for Elliptic Problems},
  North Holland, Amsterdam, 1978.

\bibitem{MR0400739}
{\sc Ph. Cl{\'e}ment}, {\em Approximation by finite element functions using
  local regularization}, Rev. Fran\c caise Automat. Informat. Recherche
  Op\'erationnelle S\'er. RAIRO Analyse Num\'erique, 9 (1975), pp.~77--84.

\bibitem{MR1472196}
{\sc D.C. Dobson and C.R. Vogel}, {\em Convergence of an iterative method for
  total variation denoising}, SIAM J. Numer. Anal., 34 (1997), pp.~1779--1791.

\bibitem{MR2575049}
{\sc V.~Duval, J.-F. Aujol, and Y.~Gousseau}, {\em The {TVL}1 model: a
  geometric point of view}, Multiscale Model. Simul., 8 (2009), pp.~154--189.

\bibitem{MR1727362}
{\sc I.~Ekeland and R.~T{\'e}mam}, {\em Convex analysis and variational
  problems}, vol.~28 of Classics in Applied Mathematics, Society for Industrial
  and Applied Mathematics (SIAM), Philadelphia, PA, english~ed., 1999.
\newblock Translated from the French.

\bibitem{MR2520163}
{\sc C.M. Elliott and S.A. Smitheman}, {\em Numerical analysis of the {TV}
  regularization and {$H^{-1}$} fidelity model for decomposing an image into
  cartoon plus texture}, IMA J. Numer. Anal., 29 (2009), pp.~651--689.

\bibitem{MR2050138}
{\sc A.~Ern and J.-L. Guermond}, {\em Theory and practice of finite elements},
  vol.~159 of Applied Mathematical Sciences, Springer-Verlag, New York, 2004.

\bibitem{MR1994316}
{\sc X.~Feng and A.~Prohl}, {\em Analysis of total variation flow and its
  finite element approximations}, M2AN Math. Model. Numer. Anal., 37 (2003),
  pp.~533--556.

\bibitem{MR2194526}
{\sc X.~Feng, M.~von Oehsen, and A.~Prohl}, {\em Rate of convergence of
  regularization procedures and finite element approximations for the total
  variation flow}, Numer. Math., 100 (2005), pp.~441--456.

\bibitem{PhysRevB.78.235401}
{\sc P.-W. Fok, R.R. Rosales, and D.~Margetis}, {\em Facet evolution on
  supported nanostructures: Effect of finite height}, Phys. Rev. B, 78 (2008),
  p.~235401.

\bibitem{MR1681462}
{\sc G.B. Folland}, {\em Real analysis}, Pure and Applied Mathematics (New
  York), John Wiley \& Sons Inc., New York, second~ed., 1999.
\newblock Modern techniques and their applications, A Wiley-Interscience
  Publication.

\bibitem{Fu1997259}
{\sc E.S. Fu, D.-J. Liu, M.D. Johnson, J.D. Weeks, and E.D. Williams}, {\em The
  effective charge in surface electromigration}, Surface Science, 385 (1997),
  pp.~259 -- 269.

\bibitem{MR2746654}
{\sc M.-H. Giga and Y.~Giga}, {\em Very singular diffusion equations: second
  and fourth order problems}, Jpn. J. Ind. Appl. Math., 27 (2010),
  pp.~323--345.

\bibitem{MR1865089}
{\sc M.-H. Giga, Y.~Giga, and R.~Kobayashi}, {\em Very singular diffusion
  equations}, in Taniguchi {C}onference on {M}athematics {N}ara '98, vol.~31 of
  Adv. Stud. Pure Math., Math. Soc. Japan, Tokyo, 2001, pp.~93--125.

\bibitem{MR1301176}
{\sc R.~Hardt and X.~Zhou}, {\em An evolution problem for linear growth
  functionals}, Comm. Partial Differential Equations, 19 (1994),
  pp.~1879--1907.

\bibitem{Kobayashi2000141}
{\sc R.~Kobayashi, J.A. Warren, and W.C. Carter}, {\em A continuum model of
  grain boundaries}, Physica D: Nonlinear Phenomena, 140 (2000), pp.~141 --
  150.

\bibitem{MR2733098}
{\sc R.V. Kohn and H.M. Versieux}, {\em Numerical analysis of a
  steepest-descent {PDE} model for surface relaxation below the roughening
  temperature}, SIAM J. Numer. Anal., 48 (2010), pp.~1781--1800.

\bibitem{MR2821259}
{\sc W.G. Litvinov, T.~Rahman, and X.-C. Tai}, {\em A modified {TV}-{S}tokes
  model for image processing}, SIAM J. Sci. Comput., 33 (2011), pp.~1574--1597.

\bibitem{TVDInterpolation}
{\sc R.H. Nochetto and A.J. Salgado}, {\em A {TV} diminishing interpolation
  operator and applications}.
\newblock Submitted, arXiv:1211.1069, 2012.

\bibitem{MR1737503}
{\sc R.H. Nochetto, G.~Savar{\'e}, and C.~Verdi}, {\em A posteriori error
  estimates for variable time-step discretizations of nonlinear evolution
  equations}, Comm. Pure Appl. Math., 53 (2000), pp.~525--589.

\bibitem{Rudin1992259}
{\sc L.I. Rudin, S.~Osher, and E.~Fatemi}, {\em Nonlinear total variation based
  noise removal algorithms}, Physica D: Nonlinear Phenomena, 60 (1992),
  pp.~259--268.

\bibitem{MR1377244}
{\sc J.~Rulla}, {\em Error analysis for implicit approximations to solutions to
  {C}auchy problems}, SIAM J. Numer. Anal., 33 (1996), pp.~68--87.

\bibitem{MR2328004}
{\sc L.~Tartar}, {\em An introduction to {S}obolev spaces and interpolation
  spaces}, vol.~3 of Lecture Notes of the Unione Matematica Italiana, Springer,
  Berlin, 2007.

\bibitem{MR2792398}
{\sc J.~Wang and B.J. Lucier}, {\em Error bounds for finite-difference methods
  for {R}udin-{O}sher-{F}atemi image smoothing}, SIAM J. Numer. Anal., 49
  (2011), pp.~845--868.

\bibitem{MR1014685}
{\sc W.P. Ziemer}, {\em Weakly differentiable functions}, vol.~120 of Graduate
  Texts in Mathematics, Springer-Verlag, New York, 1989.
\newblock Sobolev spaces and functions of bounded variation.

\end{thebibliography}

\end{document}